\documentclass[final]{siamltex}
\usepackage{amsmath}
\usepackage{amssymb}
\usepackage{url}
\usepackage{color}
\usepackage{array}
\usepackage{colortbl}
\usepackage{blkarray}
\usepackage{caption}
\usepackage{capt-of}
\usepackage{colortbl}
\usepackage{fancyhdr}
\usepackage{fancyvrb}
\usepackage{float}
\usepackage[T1]{fontenc}
\usepackage{framed}
\usepackage{geometry}
\usepackage{graphicx}
\usepackage[latin1]{inputenc}
\usepackage{stmaryrd}
\usepackage{upquote}
\usepackage[usenames,dvipsnames,svgnames,table]{xcolor}
\usepackage{xspace}
\usepackage{todonotes}
\usepackage{booktabs}
\definecolor{lightgray}{gray}{0.5}
\definecolor{darkgreen}{rgb}{0,0.6,0.13}
\newcommand{\nc}{\newcommand}
\nc{\dsp}{\displaystyle}
\nc{\txt}{\textstyle}
\nc{\reff}[1]{(\ref{#1})}
\nc{\mrm}[1]{\mathrm{#1}}
\nc{\udl}[1]{\underline{#1}}
\nc{\ovl}[1]{\overline{#1}}
\nc{\al}{\underline{\boldsymbol{\alpha}}}
\nc{\la}{\underline{\boldsymbol{\lambda}}}
\nc{\llbr}{\llbracket}
\nc{\rrbr}{\rrbracket}
\nc{\lbr}{\lbrack}
\nc{\rbr}{\rbrack}
\nc{\N}{\mathbb{N}}
\nc{\Z}{\mathbb{Z}}
\nc{\D}{\mathbb{D}}
\nc{\Q}{\mathbb{Q}}
\nc{\R}{\mathbb{R}}
\nc{\C}{\mathbb{C}}
\nc{\T}{\mathbb{T}}
\nc{\Abf}{\mathbf{A}}
\nc{\Bbf}{\mathbf{B}}
\nc{\Cbf}{\mathbf{C}}
\nc{\Ibf}{\mathbf{I}}
\nc{\Lbf}{\mathbf{L}}
\nc{\Nbf}{\mathbf{N}}
\nc{\Sbf}{\mathbf{S}}
\nc{\Mbf}{\mathbf{M}}
\nc{\Tbf}{\mathbf{T}}
\nc{\Dbf}{\mathbf{D}}
\nc{\Qbf}{\mathbf{Q}}
\nc{\Pbf}{\mathbf{P}}
\nc{\Stwo}{\mathbb{S}^2}
\nc{\tld}[1]{\tilde{#1}}
\nc{\wtld}[1]{\widetilde{#1}}
\nc{\hu}{\hat{u}}
\nc{\wh}[1]{\widehat{#1}}
\nc{\ph}{\varphi}
\nc{\sumeven}{\sum_{k=-N/2}^{N/2}{\hspace{-0.3cm}}'{\;\,}}
\nc{\sumodd}{\sum_{k=-\frac{N-1}{2}}^{\frac{N-1}{2}}}
\nc{\sumoddl}{\sum_{l=-\frac{N-1}{2}}^{\frac{N-1}{2}}}
\nc{\cqfd}{~\hbox{\vrule width 2.5pt depth 2.5 pt height 3.5 pt}}
\nc{\ra}[1]{\renewcommand{\arraystretch}{#1}}
\nc{\nf}{\normalfont}
\newcommand{\ignore}[1]{}

\title{Fourth-order time-stepping for stiff PDEs on the sphere}
\author{Hadrien Montanelli\thanks{Oxford University Mathematical Institute, Oxford OX2 6GG, UK.\ \ Supported by 
the European Research Council under the European Union's Seventh Framework Programme (FP7/2007--2013)/ERC grant agreement 
no.\ 291068.\ \ The views expressed in this article are not those of the ERC or the European Commission, and the European Union is not 
liable for any use that may be made of the information contained here.} 
\and Yuji Nakatsukasa\thanks{Oxford University Mathematical Institute, Oxford OX2 6GG, UK.\ \ Supported by JSPS as an Overseas Research Fellow.}}

\begin{document}
\setlength{\parskip}{0pt}

\maketitle

%%%%%%%%%%%%%%%%%%%%%%%%%%%%%%%%%%%%%%%%%%%%%%%%%%%%%%%%%%%%%%%%%%%%%%%%%%%
\begin{abstract}
We present in this paper algorithms for solving stiff PDEs on the unit sphere with spectral accuracy in space
and fourth-order accuracy in time.
These are based on a variant of the double Fourier sphere method in coefficient space with multiplication matrices that differ from the usual ones, 
and implicit-explicit time-stepping schemes.
Operating in coefficient space with these new matrices allows one to use a sparse direct solver, avoids the coordinate singularity and maintains smoothness 
at the poles, while implicit-explicit schemes circumvent severe restrictions on the time-steps due to stiffness.
A comparison is made against exponential integrators and it is found that implicit-explicit schemes perform best.
Implementations in MATLAB and Chebfun make it possible to compute the solution of many PDEs to high accuracy in a very convenient fashion.
\end{abstract}

%%%%%%%%%%%%%%%%%%%%%%%%%%%%%%%%%%%%%%%%%%%%%%%%%%%%%%%%%%%%%%%%%%%%%%%%%%%

%%%%%%%%%%%%%%%%%%%%%%%%%%%%%%%%%%%%%%%%%%%%%%%%%%%%%%%%%%%%%%%%%%%%%%%%%%%l

%%%%%%%%%%%%%%%%%%%%%%%%%%%%%%%%%%%%%%%%%%%%%%%%%%%%%%%%%%%%%%%%%%%%%%%%%%%
\pagestyle{myheadings}
\thispagestyle{plain}

\markboth{MONTANELLI AND NAKATSUKASA}{STIFF PDES ON THE SPHERE}

%%%%%%%%%%%%%%%%%%%%%%%%%%%%%%%%%%%%%%%%%%%%%%%%%%%%%%%%%%%%%%%%%%%%%%%%%%%
\section{Introduction}

We are interested in computing smooth solutions of stiff PDEs on the unit sphere of the form 
\begin{equation}
u_t = \mathcal{L}u + \mathcal{N}(u), \quad u(t=0,x,y,z)=u_0(x,y,z),
\label{PDE}
\end{equation}

\noindent where $u(t,x,y,z)$ is a function of time $t$ and Cartesian coordinates $(x,y,z)$ with $x^2 + y^2 + z^2=1$.
The function $u$ can be real or complex and \reff{PDE} can be a single equation, as well as a system of equations.
In this paper, we restrict our attention to $\mathcal{L} u = \alpha\Delta u$ and to a nonlinear non-differential operator $\mathcal{N}$ with constant coefficients, 
but the techniques we present can be applied to more general cases.
A large number of PDEs of interest in science and engineering take this form.
Examples on the sphere include the (diffusive) Allen--Cahn equation $u_t = \epsilon\Delta u + u - u^3$ with $\epsilon\ll1$~\cite{du2008}, the (dispersive) focusing nonlinear Schr\"{o}dinger equation $u_t=i\Delta u + iu|u|^2$~\cite{takaoka2016}, the Gierer--Meinhardt~\cite{bhattacharya2005}, Ginzburg--Landau~\cite{rubinstein1995} and Brusselator~\cite{trinh2016} equations, and many others.

There are several methods to discretize the spatial part of \reff{PDE} with spectral accuracy,
including spherical harmonics~\cite{atkinson2012}, radial basis functions (RBFs)~\cite{fornberg2015a} and the double Fourier sphere (DFS) method~\cite{merilees1973, orszag1974}.
The DFS method is the only one that leads to a $\mathcal{O}(N\log N)$ complexity per time-step, where $N$ is the total number of grid points in the spatial discretization of \reff{PDE}.
For spherical harmonics, the cost per time-step is $\mathcal{O}(N^{3/2})$ since there are no effective ``fast'' spherical transforms,\footnote{Fast $\mathcal{O}(N\log N)$ spherical transforms have received significant attention but require so far a $\mathcal{O}(N^2)$ precomputational cost~\cite{rokhlin2006, tygert2008, tygert2010}. Note that in a recent manuscript~\cite{slevinsky2017} Slevinsky proposed a new fast spherical transform based on conversions between spherical harmonics and bivariate Fourier series with a lower $\mathcal{O}(N^{3/2})$ precomputational cost. For implicit-explicit schemes with the DFS method, the precomputation is $\mathcal{O}(N)$.} and for global RBFs~\cite[Ch.~6]{fornberg2015a}, the cost per time-step is $\mathcal{O}(N^2)$ since these generate dense differentiation matrices.\footnote{RBF-FD~\cite[Ch.~7]{fornberg2015a} generate sparse matrices but only achieve algebraic orders of accuracy. Moreover, the solution time for these sparse matrices is not necessarily $\mathcal{O}(N)$.}
We focus in this paper on the DFS method and present a novel formulation operating in coefficient space.

Once the spatial part of \reff{PDE} has been discretized by the DFS method on an $n\times m$ uniform longitude-latitude grid, it becomes a system of $nm$ ODEs,
\begin{equation} 
\hat{u}' = \mathbf{L}\hat{u} + \mathbf{N}(\hat{u}), \quad \hat{u}(0)=\hat{u}_0,
\label{ODE}
\end{equation}

\noindent where $\hat{u}(t)$ is a vector of $nm$ Fourier coefficients and $\Lbf$ (an $nm\times nm$ matrix) and $\Nbf$ are the discretized versions of $\mathcal{L}$ and $\mathcal{N}$ in Fourier space. Solving the system \reff{ODE} with generic explicit time-stepping schemes can be highly challenging because of \textit{stiffness}:
the large eigenvalues of $\Lbf$---due to the second differentiation order in \reff{PDE} and the clustering of points near the poles---force one to use very small time-steps. 
Exponential integrators and implicit-explicit (IMEX) schemes are two classes of numerical methods that are aimed at treating stiffness.
For exponential integrators, the linear part~$\Lbf$ is integrated exactly using the matrix exponential while a numerical scheme is applied to the nonlinear part~$\Nbf$.
For IMEX schemes, an explicit formula is used to advance~$\Nbf$ while an implicit scheme is used to advance~$\Lbf$.
We show in this paper that the DFS method combined with IMEX schemes leads to $\mathcal{O}(nm\log nm)$ per time-step algorithms for both diffusive and dispersive PDEs, that exponential integrators achieve this complexity for diffusive PDEs only, and that IMEX schemes outperform exponential integrators in both cases.
For numerical comparisons, we consider two versions of the fourth-order ETDRK4 exponential integrator of Cox and Matthews~\cite{cox2002} and two fourth-order IMEX schemes, the IMEX-BDF4~\cite{hundsdorfer2007} and LIRK4~\cite{calvo2001} schemes.

\begin{figure}
\hspace{-1.2cm}
\includegraphics[scale=.6]{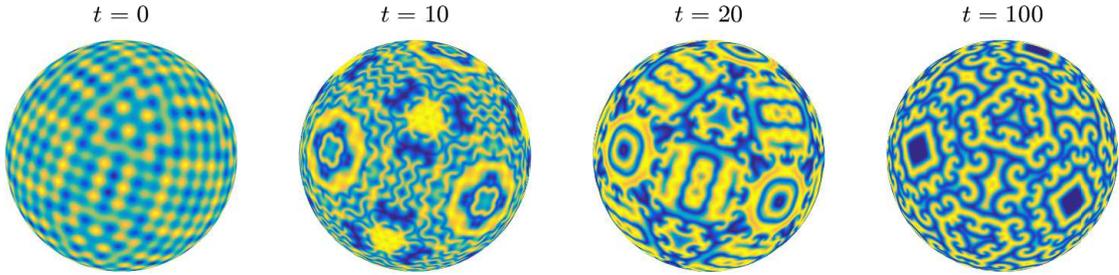}
\vspace{-10cm}
\caption{\textit{Initial condition and real part of the solution at times $t=10,20,100$ of the Ginzburg--Landau equation computed by the \textup{\texttt{spinsphere}} code. This solution is oriented in a direction at a $\pi/8$ angle from the north-south axis, so the symmetry maintained in the computations is a reflection of global accuracy.}}
\label{fig:GLsol}
\end{figure}

By contrast, Kassam and Trefethen demonstrated in~\cite{kassam2005} that exponential integrators (ETDRK4) outperform IMEX schemes (IMEX-BDF4); this can be explained by two factors.
First, they focused on problems with diagonal matrices $\Lbf$. 
(IMEX-BDF4 performed better than ETDRK4 for the only non-diagonal problem they considered.) 
For diagonal problems, exponential integrators are particularly efficient since the computation of the matrix exponential is trivial and
the matrix exponential is diagonal too (hence, its action on vectors can trivially be computed in linear time).
Second, IMEX-BDF4 is unstable for dispersive PDEs since it is based on the fourth-order backward differentiation formula, which is unstable for dispersive PDEs---this is why they could not make it work for the KdV equation. 
Our DFS method in coefficient space leads to matrices that are not diagonal but have a sparsity structure that makes IMEX schemes particularly efficient,
and we consider not only IMEX-BDF4 but also LIRK4, which is stable for dispersive PDEs.

There are libraries for solving time-dependent PDEs on the sphere, including SPHEREPACK~\cite{spherepack1999} and FEniCS~\cite{fenics2013}. 
However, none of these are aimed at solving stiff PDEs and easily allow for computing in an integrated environment.
The algorithms we shall describe in this paper are aimed at solving stiff PDEs, and have been implemented in MATLAB and made available as part of Chebfun~\cite{chebfun} in the \texttt{spinsphere} code.
(Note that \texttt{spin} stands for \textbf{s}tiff \textbf{P}DE \textbf{in}tegrator.)
The recent extension of Chebfun to the sphere~\cite{townsend2016}, built on its extension to periodic problems~\cite{montanelli2015b},
provides a very convenient framework for working with functions on the sphere. 
For example, the function
\begin{equation}
f(\lambda, \theta) = \cos(1+ \cos\lambda\sin(2\theta))
\end{equation}

\noindent can be approximated to machine precision by the following command:

\vspace{.2cm}
\begin{small}
\begin{verbatim}
     f = spherefun(@(lam,th) cos(1 + cos(lam).*sin(2*th)));
\end{verbatim}
\end{small}
\vspace{.2cm}

\noindent Using \texttt{spherefun} objects as initial conditions, the \texttt{spinsphere} code allows one to solve stiff PDEs with a few lines of code, using IMEX-BDF4 for diffusive and LIRK4 for dispersive PDEs. 
For example, the following MATLAB code solves the Ginzburg--Landau equation $u_t = 10^{-4}\Delta u + u -(1+1.5i)u|u|^2$ with $1024$ grid points in longitude and latitude and a time-step $h=10^{-1}$:

\vspace{.2cm}
\begin{small}
\begin{verbatim}  
     n = 1024;                                          % number of grid pts
     h = 1e-1;                                          % time-step
     tspan = [0 100];                                   % time interval
     S = spinopsphere(tspan);                           % initialize operator
     S.lin = @(u) 1e-4*lap(u);                          % linear part
     S.nonlin = @(u) u-(1+1.5i)*u.*abs(u).^2;           % nonlinear part  
     u0 = @(x,y,z) 1/3*(cos(40*x)+cos(40*y)+cos(40*z)); % initial cond.
     th = pi/8; c = cos(th); s = sin(th);               % pi/8 rotation
     S.init = spherefun(@(x,y,z)u0(c*x-s*z,y,s*x+c*z)); % rotated initial cond.
     u = spinsphere(S, n, h);                           % solve
\end{verbatim}
\end{small}
\vspace{.2cm}

\noindent The initial condition and the solution at times $t=10,20,100$ are shown in Figure~\ref{fig:GLsol}. 
(The Ginzburg--Landau equation in 2D goes back to the 1970s with the work of Stewartson and Stuart~\cite{stewartson1971}.
Rubinstein and Sternberg studied it on the sphere~\cite{rubinstein1995}, with applications in the study of liquid crystals~\cite{pismen1992} and nonequilibrium patterns~\cite{pomeau1983}.)
Figure~\ref{code:LIRK4} lists a detailed MATLAB code to solve the same problem; a more sophisticated version of this code is used inside \texttt{spinsphere}.
(Note that this code might be a little bit slow; for speed, the reader might want to adjust $n$ and $h$, and change the initial condition.
We also encourage the reader to type \texttt{spinsphere('gl')} or \texttt{spinsphere('nls')} to invoke an example computation.)

The paper is structured as follows. 
In the next section, we review the DFS method (Section 2.1) and present a new Fourier spectral method in coefficient space, which, 
using multiplication matrices that differ from the usual ones, avoids the coordinate singularity (Section 2.2), 
takes advantage of sparse direct solvers (Section 2.3), and maintains smoothness at the poles (Sections 2.4 and 2.5).
The time-stepping schemes are presented in Section 3 while Section 4 is dedicated to numerical comparisons on simple PDEs.

%%%%%%%%%%%%%%%%%%%%%%%%%%%%%%%%%%%%%%%%%%%%%%%%%%%%%%%%%%%%%%%%%%%%%%%%%%%
\section{A Fourier spectral method in coefficient space}

We present in this section a Fourier spectral method for the spatial discretization of \reff{PDE}, based on 
the DFS method and novel Fourier multiplication matrices in coefficient space.
The accuracy of the method is tested by solving the Poisson and heat equations.

%%%%%%%%%%%%%%%%%%%%%%%%%%%%%%%%%%%%%%%%%%%%%%%%%%%%%%%%%%%%%%%%%%%%%%%%%%%
\subsection{The double Fourier sphere method}

The DFS method uses the longitude-latitude coordinate transform,
\begin{equation}
x = \cos\lambda\sin\theta, \; y = \sin\lambda\sin\theta, \; z=\cos\theta, 
\label{coord}
\end{equation}

\noindent with $(\lambda,\theta)\in[-\pi,\pi]\times[0,\pi]$. The azimuth angle $\lambda$ corresponds to the longitude while
the polar (or zenith) angle $\theta$ corresponds to the latitude.\footnote{To be precise, $\theta$ is the colatitude, defined as ``$\pi/2$ minus latitude'' with latitude in $[-\pi/2,\pi/2]$. For simplicity, we will refer to it as latitude. Note that $\theta=0$ corresponds to the north
pole and $\theta=\pi$ to the south pole.}
A function $u(x,y,z)$ on the sphere is written as $u(\lambda,\theta)$ using \reff{coord}, i.e., 
\begin{equation}
u(\lambda, \theta) = u(\cos\lambda\sin\theta, \sin\lambda\sin\theta, \cos\theta), \quad (\lambda,\theta)\in[-\pi,\pi]\times[0,\pi],
\label{defu}
\end{equation}

\noindent and \reff{PDE} with $\mathcal{L}=\alpha\Delta$ becomes
\begin{equation}
u_t = \alpha\Delta u + \mathcal{N}(u), \quad u(t=0,\lambda,\theta)=u_0(\lambda,\theta), \quad  (\lambda,\theta)\in[-\pi,\pi]\times[0,\pi].
\label{PDE2}
\end{equation}

\noindent Note that the function $u(\lambda,\theta)$ in \reff{defu} is $2\pi$-periodic in $\lambda$ but not periodic in $\theta$.
The key idea of the DFS method---developed by Merilees~\cite{merilees1973} and further studied by Orszag~\cite{orszag1974} in the 1970's, and recently revisited by Townsend et al.\ with the use of low-rank approximations~\cite{townsend2016}---is to associate a function $\tilde{u}(\lambda,\theta)$ with $u(\lambda,\theta)$, $2\pi$-periodic in both $\lambda$ and $\theta$, defined on $[-\pi,\pi]\times[-\pi,\pi]$, and constant along the lines $\theta=0$ and $\theta=\pm\pi$ corresponding to the poles. Mathematically, the function $\tilde{u}(\lambda,\theta)$ is defined as
\begin{equation}
\tilde{u}(\lambda, \theta) = \left\{
\begin{array}{ll}
u(\lambda, \theta), & (\lambda,\theta)\in[-\pi,\pi]\times[0,\pi], \\
u(\lambda+\pi, -\theta), & (\lambda,\theta)\in[-\pi,0]\times[-\pi,0], \\
u(\lambda-\pi, -\theta), & (\lambda,\theta)\in[0,\pi]\times[-\pi,0].
\end{array}
\right.
\label{DFS}
\end{equation}

\noindent The function $u$ is ``doubled-up'' in the $\theta$-direction and flipped; see, e.g., \cite[Fig.~1]{townsend2016}.
Since the function $\tilde{u}$ is $2\pi$-periodic in both $\lambda$ and $\theta$, it can be approximated by a 2D Fourier series,
\begin{equation}
\tilde{u}(\lambda, \theta) \approx \sum_{j=-m/2}^{m/2}{\hspace{-0.3cm}}'{\;\,}\sum_{k=-n/2}^{n/2}{\hspace{-0.3cm}}'{\;\,}\hat{u}_{jk}e^{ij\theta}e^{ik\lambda}.
\label{Fourier2D}
\end{equation}

\noindent The numbers $n$ and $m$ are assumed to be even (this will be the case throughout this paper) and the primes on the summation signs mean that the boundary terms $j=\pm m/2$ or $k=\pm n/2$ are halved.
The Fourier coefficients are defined by
\begin{equation}
\hat{u}_{jk} = \frac{1}{nm}\sum_{p=1}^{m}\sum_{q=1}^{n}\tilde{u}(\lambda_q,\theta_p)e^{-ij\theta_p}e^{-ik\lambda_q}, \quad -\frac{m}{2}\leq j \leq\frac{m}{2}-1,
\quad -\frac{n}{2}\leq k \leq\frac{n}{2}-1,
\label{Coeffs2D}
\end{equation}

\noindent with $\hat{u}_{j,n/2}=\hat{u}_{j,-n/2}$ for all $j$ and $\hat{u}_{m/2,k}=\hat{u}_{-m/2,k}$ for all $k$, and correspond to a 2D uniform grid with $n$ points in longitude and $m$ points in latitude,
\begin{equation}
\lambda_q = -\pi + (q-1)\frac{2\pi}{n}, \quad 1\leq q\leq n, \quad \theta_p = -\pi + (p-1)\frac{2\pi}{m}, \quad 1\leq p\leq m.
\label{Grid2D}
\end{equation}

\noindent The $nm$ Fourier coefficients $\hat{u}_{jk}$ can be computed by sampling $\tilde{u}$ on the grid and using the 2D FFT, costing $\mathcal{O}(nm\log nm)$ operations.
In practice we take $m=n$ since it leads to the same resolution in each direction around the equator where the spacing is the coarsest.

As mentioned in~\cite{townsend2016}, every smooth function $u(\lambda,\theta)$ on the sphere is associated with a smooth bi-periodic function $\tilde{u}(\lambda,\theta)$ on $[-\pi,\pi]^2$ via~\reff{DFS},
but the converse is not true since smooth bi-periodic functions might not be constant along the lines $\theta=0$ and $\theta=\pm\pi$ corresponding to the poles.
To be smooth on the sphere, functions of the form~\reff{DFS} have to satisfy the \textit{pole conditions}, which ensures that $\tilde{u}(\lambda,\theta)$ is single-valued at the poles despite the fact that latitude circles degenerate into a single point there. For approximations of the form~\reff{Fourier2D}--\reff{Coeffs2D}, this is given by
\begin{equation}
\sum_{j=-m/2}^{m/2}{\hspace{-0.3cm}}'{\;\,}\hat{u}_{jk} = \sum_{j=-m/2}^{m/2}{\hspace{-0.3cm}}'{\;\,}(-1)^j\hat{u}_{jk} = 0, \quad |k|\geq1.
\label{polecondition1}
\end{equation}

\noindent As we will see in the numerical experiments of Sections~2.4 and 2.5, when solving a PDE involving the Laplacian operator, if the right-hand side (for Poisson's equation) or the initial condition (for the heat equation) is a smooth function on the sphere, then the solutions obtained with our DFS method are also smooth functions on the sphere. 
Therefore, we do not have to impose the conditions~\reff{polecondition1}.
Similarly, we do not impose  the ``doubled-up'' symmetry in~\reff{DFS}, as it was preserved throughout our experiments (we leave as an open problem to prove this). 
For brevity we only illustrate the condition~\reff{polecondition1} in our experiments. 

Let us finish this section with some comments about Fourier series for solving PDEs on the sphere.
One way of using them is to use standard double Fourier series on a ``doubled-up'' version of $u$, i.e., the DFS method~\reff{DFS}--\reff{Coeffs2D}.
This is what Merilees did and, combined with a Fourier spectral method in value space, he solved the shallow water equations~\cite{merilees1973}.
Another way is to use half-range cosine or sine series~\cite{boyd1978, cheong2000a, orszag1974, shen1999, yee1980}---after all, spherical harmonics are represented as proper combinations of half-ranged cosine or sine series.
For example, Orszag~\cite{orszag1974} suggested the use of approximations of the form
\begin{equation}
u(\lambda, \theta) \approx \sum_{k=-n/2}^{n/2}\sum_{j=0}^{n/2}\hat{u}_{jk}\sin^s\theta\cos j\theta e^{ik\lambda},
\label{Orszag}
\end{equation}

\noindent where $s=0$ if $k$ is even, $s=1$ if $k$ is odd.
(Note that the Fourier series in~\reff{Orszag} approximates $u$ directly, as opposed to $\tilde{u}$.)
When using half-range cosine or sine terms as basis functions, one must be careful that the pole conditions are satisfied.
One can either select basis functions that satisfy the pole conditions or impose a constraint on the Fourier coefficients to enforce it.
For solving PDEs involving the Laplacian operator (in both value and coefficient spaces), Orszag imposed certain constraints on the coefficients $\hat{u}_{jk}$ in~\reff{Orszag}, 
analogous to those in~\reff{polecondition1},
\begin{equation}
\sum_{j=0}^{n/2} \hat{u}_{jk} = \sum_{j=0}^{n/2} (-1)^j\hat{u}_{jk} = 0, \quad |k|\geq2.
\label{polecondition2}
\end{equation}

\noindent Boyd~\cite{boyd1978} studied Orszag's method and showed that the constraints~\reff{polecondition2} are actually not necessary
for solving \textit{time-independent} PDEs in value space but mentioned that the ``absence of pole constraints is still risky'' when working with coefficients.

The spectral method we present in this paper is based on Merilees' approach and is similar to the method of Townsend et al.~\cite{townsend2016},
but the Fourier multiplication matrices we use are different.
It is simpler to implement than the half-range cosine/sine methods since there are no constraints to impose, and gives comparable accuracy.

%%%%%%%%%%%%%%%%%%%%%%%%%%%%%%%%%%%%%%%%%%%%%%%%%%%%%%%%%%%%%%%%%%%%%%%%%%%
\subsection{Fourier multiplication matrices in coefficient space}

In this section, we are interested in finding a matrix for multiplication by $\sin^2\theta$ that is nonsingular. This is crucial because some of the time-stepping methods we shall describe
in Section~3 need to compute the inverse of such a matrix. For this discussion we can restrict our attention to the 1D case.

Consider an even number $m$ of equispaced points $\{\theta_p\}_{p=1}^m$ on $[-\pi,\pi]$,
\begin{equation}
\theta_p = -\pi + (p-1)\frac{2\pi}{m}, \quad 1\leq p \leq m.
\end{equation}

\noindent (Note that these points include $-\pi$ but not $\pi$.)
Let $u$ be a complex-valued function on $[-\pi,\pi]$ with values $\{u_p\}_{p=1}^m$ at these points.
It is well known~\cite[Ch.~13]{henrici1986} that there exists a unique degree $m/2$ trigonometric polynomial $p(\theta)$ that interpolates $u(\theta)$ at these $m$ points,
i.e., such that $p(\theta_p)=u_p$ for each $p$, of the symmetric form
\begin{equation}
p(\theta) = \sum_{j=-m/2}^{m/2}{\hspace{-0.3cm}}'{\;\,}\hat{u}_j e^{ij\theta}, 
\label{triginterp}
\end{equation}

\noindent with Fourier coefficients
\begin{equation}
\hat{u}_j = \frac{1}{m}\sum_{p=1}^{m}u_pe^{-ij\theta_p}, \quad -\frac{m}{2}\leq j\leq\frac{m}{2}-1,
\end{equation}

\noindent and $\hat{u}_{m/2}=\hat{u}_{-m/2}$. 
The prime on the summation sign indicates that the terms $j=\pm m/2$ are halved.
Hence, we shall define the vector of $m+1$ Fourier coefficients as
\begin{equation}
\hat{u} = \Big(\frac{\hat{u}_{-m/2}}{2}, \hat{u}_{-m/2+1},\ldots,\hat{u}_{m/2-1},\frac{\hat{u}_{m/2}}{2}=\frac{\hat{u}_{-m/2}}{2}\Big)^T.
\label{vec1}
\end{equation}

\noindent Let us emphasize that if the trigonometric interpolant were defined as 
\begin{equation}
p(\theta) = \sum_{j=-m/2}^{m/2-1}\hat{u}_j e^{ij\theta},
\label{triginterp2}
\end{equation}

\noindent the derivative of \reff{triginterp2} would have a mode $(-im/2)e^{-im\theta/2}$ leading to complex values for real data.\footnote{Consider for example $u(\theta)=\cos(\theta)$ with $m=2$. The
representation \reff{triginterp2} gives $p(\theta)=e^{-i\theta}$ with correct values $-1$ and $1$ at grid points $\theta_1=-\pi$ and $\theta_2=0$ but its derivative $p'(\theta)=-ie^{-i\theta}$ is complex-valued on the grid. 
The representation \reff{triginterp} gives $p(\theta)=1/2(e^{-i\theta}+e^{i\theta})$, which is indeed the correct answer.}
However, FFT codes only store $m$ coefficients, i.e., they assume that $p(\theta)$ is of the form~\reff{triginterp2} with
\begin{equation}
\hat{u} = \Big(\frac{\hat{u}_{-m/2}}{2}+\frac{\hat{u}_{m/2}}{2}=\hat{u}_{-m/2}, \hat{u}_{-m/2+1},\ldots,\hat{u}_{m/2-1}\Big)^T.
\label{vec2}
\end{equation}

\noindent As a consequence, the first entry of the $m\times m$ first-order Fourier differentiation matrix $\Dbf_m$, which acts on \reff{vec2},
is zero,
\begin{equation}
\Dbf_m = \mathrm{diag}\Big(i(0,-m/2+1,-m/2+2,\ldots,m/2-1)\Big),
\label{diffmat}
\end{equation}

\noindent to cancel the mode $(-im/2)e^{-im\theta/2}$.
Another way of seeing this is to adopt the following point of view: to compute derivatives, we map the vector of $m$ Fourier coefficients \reff{vec2} to 
the representation \reff{vec1} with $m+1$ coefficients, differentiate, and then map back to $m$ coefficients.
Thus, $\Dbf_m$ can be written as the product of three matrices,\footnote{Readers might find details such as \reff{Deven}--\reff{Qmat}
unexciting, and we would not disagree. But what trouble it causes in computations if you do not get these details right!}
\begin{equation}
\Dbf_m = \Qbf\Dbf_{m+1}\Pbf
\label{Deven}
\end{equation}

\noindent where the $(m+1)\times m$ matrix $\Pbf$ maps \reff{vec2} to \reff{vec1},
\begin{equation}
\Pbf  = \begin{pmatrix}
\frac{1}{2} \\
& 1 \\
& & \ddots \\
& & & 1 \\
\frac{1}{2} & & & 0
\end{pmatrix},
\label{Pmat}
\end{equation}

\noindent $\Dbf_{m+1}$ is the $(m+1)\times(m+1)$ first-order Fourier differentiation matrix,
\begin{equation}
\Dbf_{m+1} = \mathrm{diag}\Big(i(-m/2, -m/2+1, \ldots, m/2)\Big),
\label{Dodd}
\end{equation}

\noindent and $\Qbf$ is the $m\times(m+1)$ matrix that maps back to $m$ coefficients,
\begin{equation}
\Qbf = \begin{pmatrix}
1 & & & & 1\\
& 1 \\
& & \ddots \\
& & & 1 & 0 
\end{pmatrix}.
\label{Qmat}
\end{equation}

\noindent (Note that the first entry of the differentiation matrix \reff{Dodd} is nonzero.)

The same point of view can be adopted for multiplication matrices, with the difference that multiplying by $\sin^2\theta$ or $\cos\theta\sin\theta$ 
will increase the length of the representation by four since
\begin{equation}
\sin^2\theta = -\frac{1}{4}e^{-2i\theta} + \frac{1}{2} - \frac{1}{4}e^{2i\theta}, \quad \cos\theta\sin\theta = -\frac{1}{4}e^{-2i\theta} + \frac{1}{4}e^{2i\theta}.
\end{equation}

\noindent Therefore, to multiply by, e.g., $\sin^2\theta$, we map \reff{vec2} to \reff{vec1}, multiply by $\sin^2\theta$ with an $(m+1+4)\times(m+1)$ matrix, 
and then truncate and map back to $m$ coefficients. The resulting matrix for multiplication by $\sin^2\theta$, which we denote by $\Tbf_{\sin^2}$, is given by
\begin{equation}
\Tbf_{\sin^2} = \Qbf\Mbf_{\sin^2}(:,3:m+3)\Pbf,
\label{Tsin2_a}
\end{equation}

\noindent where $\Mbf_{\sin^2}$ is the $(m+1+4)\times(m+1+4)$ matrix defined by
\begin{equation}
\Mbf_{\sin^2} = \begin{pmatrix}
\frac{1}{2} & 0 & -\frac{1}{4} \vphantom{\ddots} \\
0 & \frac{1}{2} & 0 & -\frac{1}{4} \\
-\frac{1}{4} & 0 & \ddots & \ddots & \ddots \\
& -\frac{1}{4} & \ddots & \ddots & \ddots & \ddots \\
& & \ddots & \ddots & \ddots & \ddots & -\frac{1}{4} \\
& & & \ddots & \ddots & \ddots & 0 \\
& & & & -\frac{1}{4} & 0 & \frac{1}{2} \vphantom{\ddots} 
\end{pmatrix},
\label{Msin2}
\end{equation}

\noindent $\Pbf$ is defined as before and $\Qbf$ is the following $m\times(m+1+4)$ matrix,
\begin{equation}
\Qbf = \begin{pmatrix}
0 & 0 & 1 & & & & 1 & 0 & 0\\
& & & 1 \\
& & & & \ddots \\
& & & & & 1 & 0 & 0 & 0
\end{pmatrix}.
\end{equation}

\noindent We have used MATLAB notation in \reff{Tsin2_a}: $\Mbf_{\sin^2}(:,3:m+3)$ is obtained from \reff{Msin2} by removing the first and last two columns---these
columns would hit zero coefficients in the padded-with-zeros version of $\hat{u}$. This leads to
\begin{equation}
\Tbf_{\sin^2} = \begin{pmatrix}
\frac{1}{2} & 0 & -\frac{1}{4} & & & & -\frac{1}{4} & 0 \vphantom{\ddots} \\
0 & \frac{1}{2} & 0 & -\frac{1}{4} & & & & 0 \\
-\frac{1}{8} & 0 & \ddots & \ddots & \ddots \\
& -\frac{1}{4} & \ddots & \ddots & \ddots & \ddots \\
& & -\frac{1}{4} & \ddots & \ddots & \ddots & -\frac{1}{4} \\
& & & \ddots & \ddots & \ddots & \ddots & -\frac{1}{4} \\
-\frac{1}{8} & & & & -\frac{1}{4} & \ddots & \frac{1}{2} & 0 \\
0 & 0 & & & & -\frac{1}{4} & 0 & \frac{1}{2} \vphantom{\ddots}
\end{pmatrix}.
\label{Tsin2_b}
\end{equation}

\noindent Using the Gershgorin circle theorem~\cite{varga2004} we see that the $m\times m$ matrix \reff{Tsin2_b} is nonsingular since it is row diagonally dominant, with strict diagonal dominance in the second row, and irreducible.

Let us add some comments about \reff{Tsin2_b}. 
If we operated in value space, we would obtain a singular matrix, since the multiplication matrix in value space, 
$\Mbf_{\sin^2}^v$, a diagonal matrix with entries $\{\sin^2\theta_p\}_{p=1}^m$, has two zeros corresponding to 
$\theta_p=-\pi$ and $\theta_p=0$.
(The standard remedy in that case is to shift the $\theta$-grid so that it does not contain the poles~\cite{merilees1973}.)
From this matrix, we can obtain a multiplication in coefficient space by multiplying by the DFT matrix $\textbf{F}$ and its inverse, 
\begin{equation}
\textbf{F}\Mbf_{\sin^2}^v\textbf{F}^{-1} = \begin{pmatrix}
\frac{1}{2} & 0 & -\frac{1}{4} & & & & -\frac{1}{4} & 0 \vphantom{\ddots} \\
0 & \frac{1}{2} & 0 & -\frac{1}{4} & & & & -\frac{1}{4} \\
-\frac{1}{4} & 0 & \ddots & \ddots & \ddots \\
& -\frac{1}{4} & \ddots & \ddots & \ddots & \ddots \\
& & \ddots & \ddots & \ddots & \ddots & -\frac{1}{4} \\
& & & \ddots & \ddots & \ddots & \ddots & -\frac{1}{4} \\
-\frac{1}{4} & & & & -\frac{1}{4} & \ddots & \frac{1}{2} & 0 \\
0 & -\frac{1}{4} & & & & -\frac{1}{4} & 0 & \frac{1}{2} \vphantom{\ddots}
\end{pmatrix}.
\label{Msin2t}
\end{equation}

\noindent This matrix is indeed singular since $\textbf{F}$ defines a unitary transformation, and its null space contains the vectors $(1,1,\ldots)^T$ and $(1,-1,1,-1,\ldots)^T$,
which correspond to the Fourier coefficients of the delta functions at $\theta=0$ and $\theta=-\pi$.

In~\cite{townsend2016}, the authors use the $m\times m$ version of \reff{Msin2} (which is also nonsingular) as opposed to \reff{Tsin2_b}; 
this leads to incorrect results for trigonometric polynomials of degree $m/2-2$.
To illustrate this, let us consider $m=6$ and multiply $\sin^2(\theta)$ by a trigonometric polynomial of degree $m/2-2=1$, e.g., $\cos(\theta)$.
In the representation~\reff{vec2}, the function $\cos(\theta)=1/2(e^{-i\theta}+e^{i\theta})$ has coefficients
\begin{equation}
c = \Big(0, 0, \frac{1}{2}, 0, \frac{1}{2}, 0\Big)^T,
\end{equation}

\noindent while the product $\cos(\theta)\sin(\theta)=-1/8(e^{-3i\theta}+e^{3i\theta}) + 1/8(e^{-i\theta}+e^{i\theta})$ has coefficients
\begin{equation}
d = \Big(-\frac{1}{4}, 0, \frac{1}{8}, 0, \frac{1}{8}, 0\Big)^T.
\end{equation}

\noindent For $m=6$, we have
\begin{equation}
\Tbf_{\sin^2} = \begin{pmatrix} 
\frac{1}{2} & 0 & -\frac{1}{4} & 0 & -\frac{1}{4} & 0 \\
0 & \frac{1}{2} & 0 &  -\frac{1}{4} & 0 & 0 \\
-\frac{1}{8} & 0 & \frac{1}{2} & 0 & -\frac{1}{4} & 0 \\
0 & -\frac{1}{4} & 0 & \frac{1}{2} & 0 & -\frac{1}{4} \\
-\frac{1}{8} & 0 & -\frac{1}{4} & 0 & \frac{1}{2} & 0 \\
0 & 0 & 0 & -\frac{1}{4} & 0 & \frac{1}{2}
\end{pmatrix}, \quad
\Tbf_{\sin^2}\,c = \Big(-\frac{1}{4}, 0, \frac{1}{8}, 0, \frac{1}{8}, 0\Big)^T = d,
\end{equation}
\noindent while
\begin{equation}
\Mbf_{\sin^2} = \begin{pmatrix} 
\frac{1}{2} & 0 & -\frac{1}{4} & 0 & 0 & 0 \\
0 & \frac{1}{2} & 0 &  -\frac{1}{4} & 0 & 0 \\
-\frac{1}{4} & 0 & \frac{1}{2} & 0 & -\frac{1}{4} & 0 \\
0 & -\frac{1}{4} & 0 & \frac{1}{2} & 0 & -\frac{1}{4} \\
0 & 0 & -\frac{1}{4} & 0 & \frac{1}{2} & 0 \\
0 & 0 & 0 & -\frac{1}{4} & 0 & \frac{1}{2}
\end{pmatrix}, \quad
\Mbf_{\sin^2}\,c = \Big(-\frac{1}{8}, 0, \frac{1}{8}, 0, \frac{1}{8}, 0\Big)^T \neq d.
\end{equation}

Similarly, the matrix for multiplication by $\cos\theta\sin\theta$ is the product of three matrices,
\begin{equation}
\Tbf_{\cos\sin} = \Qbf\Mbf_{\cos\sin}(:,3:m+3)\Pbf,
\end{equation}

\noindent with
\begin{equation}
\Mbf_{\cos\sin} = \begin{pmatrix}
0 & 0 & \frac{i}{4} \vphantom{\ddots} \\
0 & 0 & 0 & \frac{i}{4} \\
-\frac{i}{4} & 0 & \ddots & \ddots & \ddots \\
& -\frac{i}{4} & \ddots & \ddots & \ddots & \ddots \\
& & \ddots & \ddots & \ddots & \ddots & \frac{i}{4} \\
& & & \ddots & \ddots & \ddots & 0 \\
& & & & -\frac{i}{4} & 0 & 0 \vphantom{\ddots}
\end{pmatrix},
\end{equation}

\noindent and is given by
\begin{equation}
\Tbf_{\cos\sin} = \begin{pmatrix}
0 & 0 & \frac{i}{4} & & & & -\frac{i}{4} & 0 \vphantom{\ddots} \\
0 & 0 & 0 & \frac{i}{4} & & & & 0 \\
-\frac{i}{8} & 0 & \ddots & \ddots & \ddots \\
& -\frac{i}{4} & \ddots & \ddots & \ddots & \ddots \\
& & -\frac{i}{4} & \ddots & \ddots & \ddots & \frac{i}{4} \\
& & & \ddots & \ddots & \ddots & \ddots & \frac{i}{4} \\
\frac{i}{8} & & & & -\frac{i}{4} & \ddots & 0 & 0 \\
0 & 0 & & & & -\frac{i}{4} & 0 & 0 \vphantom{\ddots}
\end{pmatrix}.
\end{equation}

%%%%%%%%%%%%%%%%%%%%%%%%%%%%%%%%%%%%%%%%%%%%%%%%%%%%%%%%%%%%%%%%%%%%%%%%%%%
\subsection{Laplacian matrix and linear systems}

The Laplacian operator on the sphere is 
\begin{equation}
\Delta u = \frac{1}{\sin\theta}\big(\sin\theta\,u_\theta\big)_\theta + \frac{1}{\sin^2\theta}u_{\lambda\lambda},
\end{equation}

\noindent which we write as
\begin{equation}
\Delta u = u_{\theta\theta} + \frac{\cos\theta\sin\theta}{\sin^2\theta}u_\theta + \frac{1}{\sin^2\theta}u_{\lambda\lambda}.
\label{lap}
\end{equation}

\noindent We want to discretize $\Delta$ with a matrix $\Lbf$ using a Fourier spectral method in coefficient space on an $n\times m$ uniform
longitude-latitude grid~\reff{Grid2D}, and we look for a solution of the form~\reff{Fourier2D}--\reff{Coeffs2D}.
Using Kronecker products, we can write $\Lbf$ as
\begin{equation}
\Lbf = \Ibf_n\otimes(\Dbf_{m}^{(2)} +  \Tbf_{\sin^2}^{-1} \Tbf_{\cos\sin} \Dbf_{m}) + \Dbf_{n}^{(2)}\otimes(\Tbf_{\sin^2}^{-1}),
\label{Laplacian}
\end{equation}

\noindent where $\Tbf_{\sin^2}$, $\Tbf_{\cos\sin}$ and $\Dbf_m$ have been defined in the previous section, $\Ibf_n$ is the $n\times n$ identity matrix and 
$\Dbf_{m}^{(2)}$ is the second-order Fourier differentiation matrix,
\begin{equation}
\Dbf^{(2)}_m = \mathrm{diag}\Big(-(m/2)^2, -(m/2-1)^2, \ldots, -1, 0, -1, \ldots, -(m/2-1)^2\Big).
\label{diffmat2}
\end{equation}

\noindent Note that the matrix $\Lbf$ is block diagonal with $n$ dense blocks of size $m\times m$. 
Let us emphasise that the $n$ blocks correspond to the $n$ longitudinal wavenumbers $-n/2\leq k\leq n/2-1$ and that the size~$m$ of each block corresponds to the $m$ latitudinal wavenumbers $-m/2\leq j\leq m/2-1$.

\begin{figure}
\hspace{-.5cm}
\begin{minipage}[t]{0.32\hsize}
\includegraphics[height=0.93\textwidth]{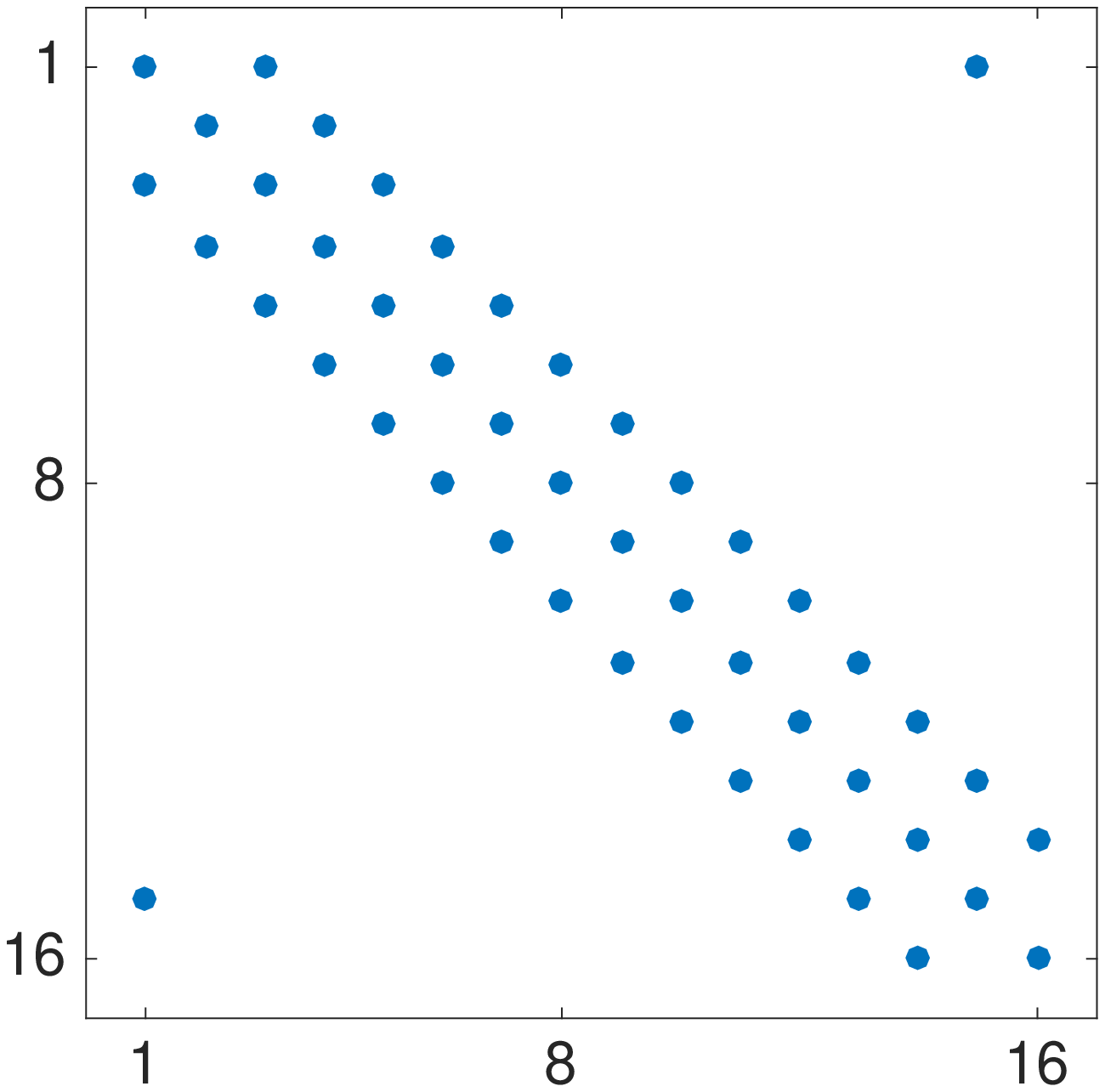}
\end{minipage}   
\begin{minipage}[t]{0.32\hsize}
\includegraphics[height=0.93\textwidth]{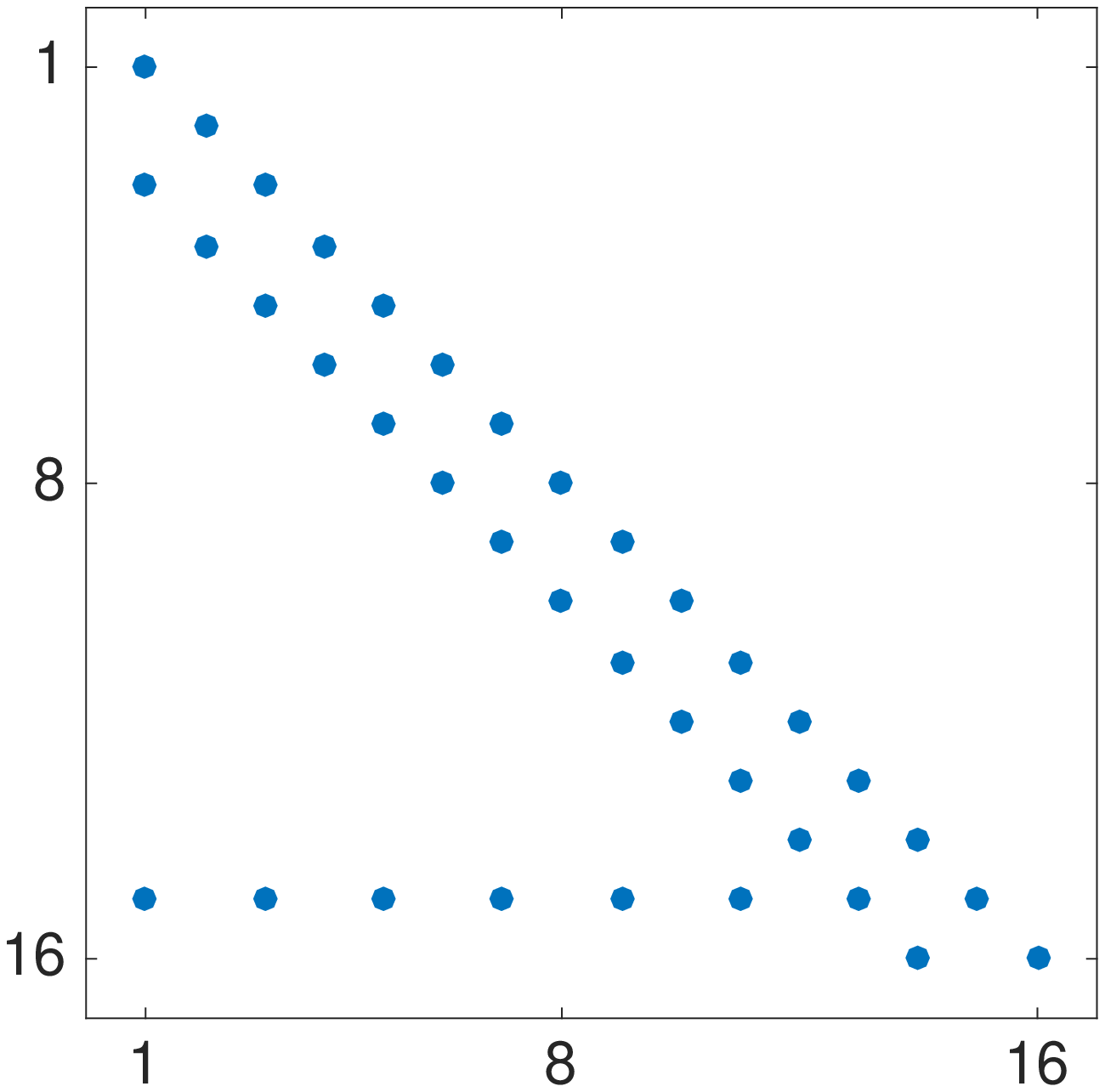}
\end{minipage}
\begin{minipage}[t]{0.32\hsize}
\includegraphics[height=0.93\textwidth]{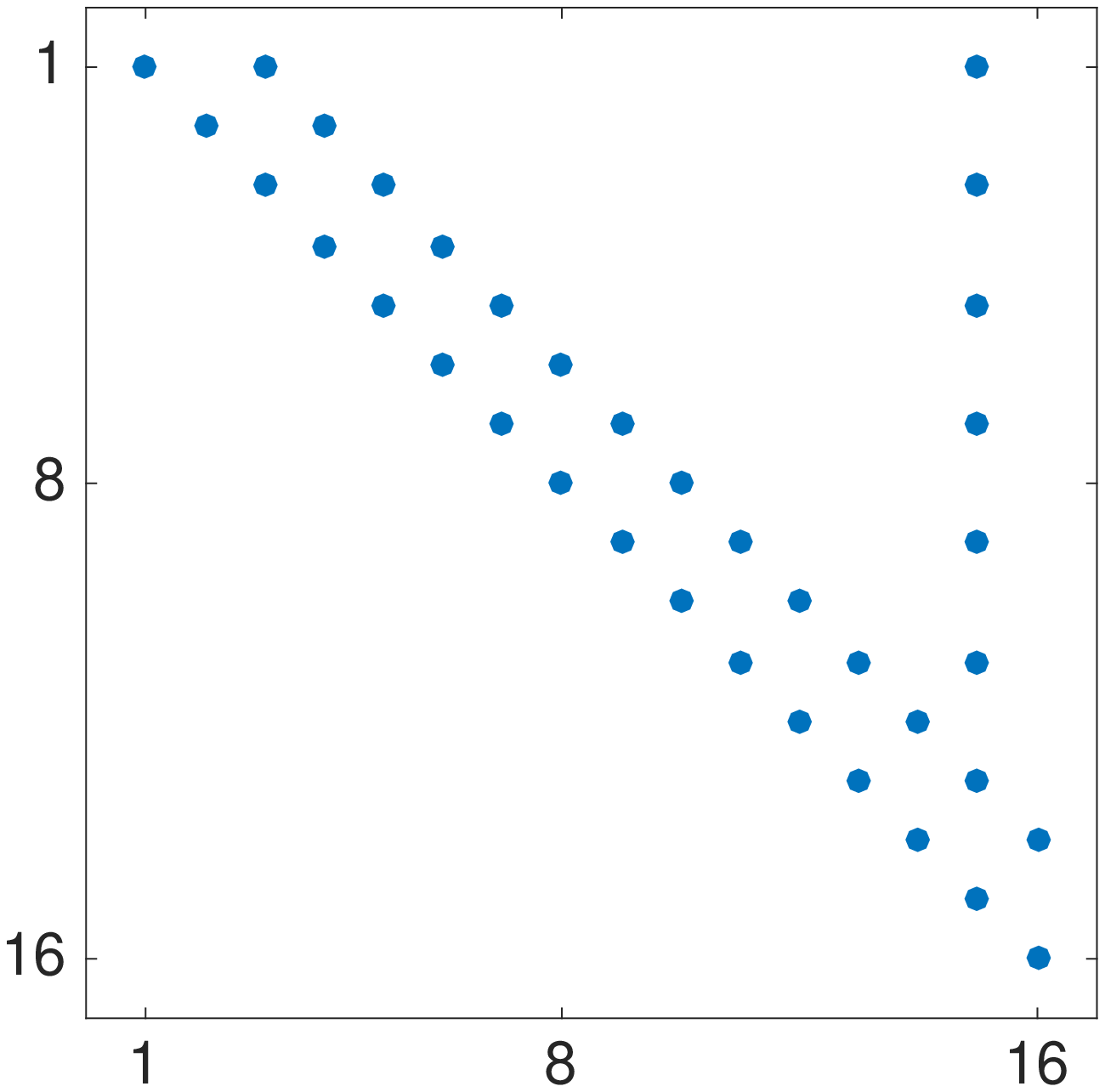}
\end{minipage}
\caption{\textit{Sparsity pattern of the matrix $z\Tbf_{\sin^2} + w\Tbf_{\sin^2}\Lbf_i$ (left), and its $\mrm{L}$ (middle) and $\mrm{U}$ (right) factors for $m=16$. Triangular systems involving $\mrm{L}$ and $\mrm{U}$ are solvable in $\mathcal{O}(m)$ operations.}
}
\label{fig:LUspyplots}
\end{figure}

Some of the time-stepping schemes we shall describe in Section 3 involve solving linear systems of the form $(z\Ibf_{nm} + w\Lbf)x=b$,
where $\Ibf_{nm}$ denotes the $nm\times nm$ identity matrix.
Fortunately, the matrix structure allows for a linear-cost direct solver.
The key observation is that $\Lbf$ is block diagonal, and each $m\times m$ block $\Lbf_i$ of $\Lbf$,
\begin{equation}
\Lbf_i = \Dbf_m^{(2)} + \Tbf_{\sin^2}^{-1}\Tbf_{\cos\sin}\Dbf_m + \Dbf^{(2)}_n(i,i)\Tbf_{\sin^2}^{-1},
\end{equation}

\noindent is dense but $(z\Ibf_m + w\Lbf_i)x = b$ can be solved in $\mathcal{O}(m)$ operations since it is equivalent to solving
\begin{equation}
(z\Tbf_{\sin^2} + w\Tbf_{\sin^2}\Lbf_i)x = \Tbf_{\sin^2}b,
\label{eq:shiftlin}
\end{equation}

\noindent and $(z\Tbf_{\sin^2} + w\Tbf_{\sin^2}\Lbf_i)$ is pentadiagonal with two (near-)corner elements. 
Therefore, using a sparse direct solver based on the standard LU factorization without pivots~\cite{davis2006}, 
the L and U factors have the sparsity patterns indicated in Figure~\ref{fig:LUspyplots} (due to the diagonal dominance in most blocks, the LU factorization without pivoting completes 
without breaking down\footnote{For the IMEX schemes of Section~3.2, we can prove that all the blocks but one are diagonally dominant; for ETDRK4-CF (see Section~3.1.1), the diagonal
dominance is violated much more frequently. Nonetheless, diagonal dominance is merely a sufficient condition for the LU factorization to not require pivoting, and in practice, 
all the methods result in linear systems for which the LU factorization causes no stability issues.}). 
Since L and U have at most three nonzero elements per column and row respectively, each triangular linear system is solvable in $\mathcal{O}(m)$ operations; 
thus once an LU factorization is computed, the linear system~\eqref{eq:shiftlin} can be solved in $\mathcal{O}(m)$ operations. 
Therefore, linear systems of the form
\begin{equation}
(z\Ibf_{nm} + w\Lbf)x = b
\label{eq:shiftlinsysL}
\end{equation}

\noindent can be solved blockwise in $\mathcal{O}(nm)$ operations.\footnote{In practice we do not solve linear systems of the form~\reff{eq:shiftlinsysL} blockwise, 
i.e., by solving $n$ linear systems~\reff{eq:shiftlin}.
Instead, a more efficient approach is to store the left-hand sides of~\reff{eq:shiftlin} altogether as a sparse matrix and use MATLAB's sparse linear solver.}
Because of the structure of the LU factors (see Figure~\ref{fig:LUspyplots}), the coefficients one gets when solving systems of the form~\reff{eq:shiftlinsysL} have the 
property that the even modes in $\lambda$ correspond to even functions in $\theta$ and the odd modes to odd functions.
To see this, note that every other term in the LU factors is zero, so the even and odd modes are decoupled.

%%%%%%%%%%%%%%%%%%%%%%%%%%%%%%%%%%%%%%%%%%%%%%%%%%%%%%%%%%%%%%%%%%%%%%%%%%%
\subsection{Poisson's equation}

To test the accuracy of \reff{Laplacian}, we first solve a \textit{time-independent} PDE, \textit{Poisson's equation}, with a zero-mean condition for uniqueness,
\begin{equation}
\begin{array}{l}
\Delta u=f(\lambda,\theta), \quad (\lambda,\theta)\in[-\pi,\pi]\times[0,\pi], \\\\
\dsp\int_{0}^{\pi}\int_{-\pi}^{\pi} u(\lambda, \theta)\sin\theta d\lambda d\theta = 0,
\end{array}
\label{Poisson}
\end{equation}

\noindent where $f$ also has zero mean on $[-\pi,\pi]\times[0,\pi]$. Using the DFS method, we seek a solution $\tilde{u}$ of the ``doubled-up'' version of \reff{Poisson},
\begin{equation}
\begin{array}{l}
\Delta \tilde{u}=\tilde{f}(\lambda,\theta), \quad (\lambda,\theta)\in[-\pi,\pi]^2, \\\\
\dsp\int_{0}^{\pi}\int_{-\pi}^{\pi} \tilde{u}(\lambda, \theta)\sin\theta d\lambda d\theta = 0,
\end{array}
\label{Poisson2}
\end{equation}

\noindent of the form~\reff{Fourier2D}--\reff{Coeffs2D}.
The true solution $u$ can be recovered by restricting $\tilde{u}$ to $(\lambda,\theta)\in[-\pi,\pi]\times[0,\pi]$. 
(Note that the zero-mean condition in \reff{Poisson2} is on the original domain $[-\pi,\pi]\times[0,\pi]$ since $\tilde{u}$ must coincide with $u$ on this domain.)
Townsend et al.~\cite{townsend2016} showed that the zero-mean condition can be discretized as
\begin{equation}
\begin{array}{ll}
\dsp \int_{0}^{\pi}\int_{-\pi}^{\pi} \tilde{u}(\lambda, \theta)\sin\theta d\lambda d\theta 
& \dsp \approx \sum_{j=-m/2}^{m/2}{\hspace{-0.3cm}}'{\;\,}\sum_{k=-n/2}^{n/2}{\hspace{-0.3cm}}'{\;\,}\hat{u}_{jk}\int_0^\pi\sin\theta e^{ij\theta}d\theta\int_{-\pi}^\pi e^{ik\lambda}d\lambda \\\\
& \dsp = 2\pi\sum_{j=-m/2}^{m/2}{\hspace{-0.3cm}}'{\;\,}\hat{u}_{j0}\frac{1 + e^{ij\pi}}{1-j^2}=0.
\end{array}
\label{zeromean}
\end{equation}

\noindent Poisson's equation \reff{Poisson2} is then discretized by
\begin{equation}
\Lbf\hat{u} = \hat{f},
\label{Poisson3}
\end{equation}

\noindent where $\hat{u}$ and $\hat{f}$ are the vectors of $nm$ Fourier coefficients~\reff{Coeffs2D} of $\tilde{u}$ and $\tilde{f}$, and $\Lbf$ is the Laplacian matrix \reff{Laplacian}.
We impose the zero-mean condition by replacing the $(m/2+1)$st row of the $(n/2+1)$st block of~$\Lbf$ by \reff{zeromean}.\footnote{The $(n/2+1)$st block of~$\Lbf$ corresponds to longitudinal wavenumber $k=0$ while the $(m/2+1)$st row of this block corresponds to latitudinal wavenumber $j=0$. }
Note that, given the Fourier coefficients of $\tilde{u}$ and $\tilde{f}$, \reff{Poisson3} can be solved in $\mathcal{O}(nm)$ operations since it is of the form~\reff{eq:shiftlinsysL} with $z=0$.

Since the Laplacian operator on the sphere has real eigenvalues $-l(l+1)$, $l\geq 0$, with eigenfunctions the spherical harmonics $Y^m_l(\lambda,\theta)$~\cite{atkinson2012}, a simple test is to take the right-hand side $f$ to be a spherical harmonic $Y_{l}^{m}$ with exact solution $u=-1/(l(l+1))Y_{l}^{m}$.
A slightly more complicated test is given in \cite{cheong2000a} and this is what we shall present here. 
We solve Poisson's equation for a family of right-hand sides $f_l$ defined by
\begin{equation}
f_l(\lambda,\theta) = l(l+1)\sin^l\theta\cos (l\lambda) +  (l+1)(l+2)\cos\theta\sin^l\theta\cos (l\lambda), \quad l\geq1.
\end{equation}

\noindent The exact solution $u_l^{ex}(\lambda,\theta)$ is given by
\begin{equation}
u_l^{ex}(\lambda,\theta) = -\sin^l\theta\cos (l\lambda) - \cos\theta\sin^l\theta\cos (l\lambda), \quad l\geq1.
\end{equation}

\noindent We compute the solutions $u_l$ for $m=n=128$ grid points in each direction for $1\leq l\leq 64$ and, following \cite[Fig.~1]{cheong2000a}, we plot the logarithm (in base 10) of the relative $L^2$-error $E$,
\begin{equation}
E = \frac{||u_l(\lambda,\theta) - u_l^{ex}(\lambda,\theta)||_2}{||u_l^{ex}(\lambda,\theta)||_2},
\label{error_Poisson}
\end{equation}

\noindent with (continuous) $L^2$-norm on the sphere $||\cdot||_2$, against $l$.
(The $L^2$-norm of a \texttt{spherefun} can be computed in Chebfun with the \texttt{norm} command.)
We also plot the logarithm of the error $P$ in satisfying the pole condition~\reff{polecondition1}, i.e.,
\begin{equation}
P = \max\Bigg(\max_{k\neq 0}\Bigg|\sum_{j=-m/2}^{m/2}{\hspace{-0.3cm}}'{\;\,}\hat{u}_{jk}^l\Bigg|, \max_{k\neq 0}\Bigg|\sum_{j=-m/2}^{m/2}{\hspace{-0.3cm}}'{\;\,}(-1)^j\hat{u}_{jk}^l\Bigg|\Bigg),
\label{error_polecondition}
\end{equation}

\noindent where the $\hat{u}_{jk}^l$ are the computed coefficients of $\tilde{u}_l$; see Figure \ref{fig:poisson}. 
The accuracy is excellent; the results are similar to those shown in~\cite[Fig.~1]{cheong2000a} and to results we have obtained with the Poisson solver of~\cite{townsend2016}
(although the discretization matrices are different as noted in Section 2.2, the effect on the solution is negligible when $m$ and $n$ are taken large enough).

\begin{figure}
\centering
\includegraphics[scale=.5]{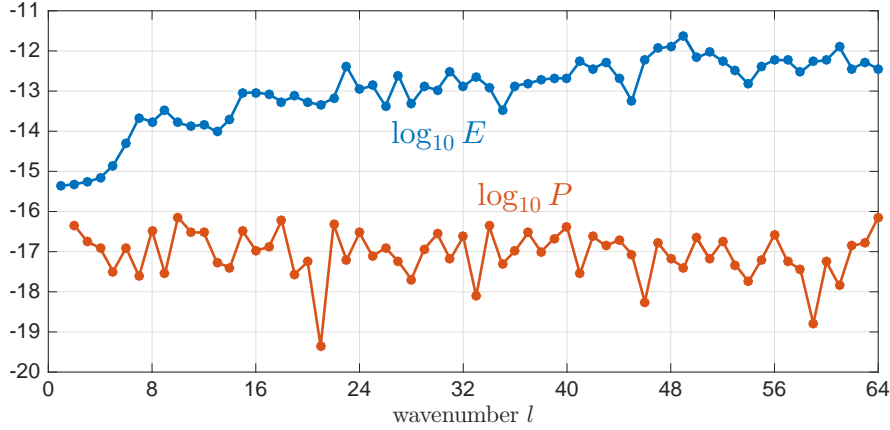}
\caption{\textit{Variation of the relative error $\log_{10}E$ and of the error in the pole condition $\log_{10}P$ with wavenumber $l$ for $m=n=128$.  
The accuracy is excellent for every wavenumber $1\leq l\leq 64$.}}
%Note that $l=64$ corresponds to the highest wavenumber that can be resolved on a $128\times 128$ grid.}}
\label{fig:poisson}
\end{figure}

%%%%%%%%%%%%%%%%%%%%%%%%%%%%%%%%%%%%%%%%%%%%%%%%%%%%%%%%%%%%%%%%%%%%%%%%%%%
\subsection{Heat equation}

\begin{figure}
\centering
\includegraphics[scale=.4]{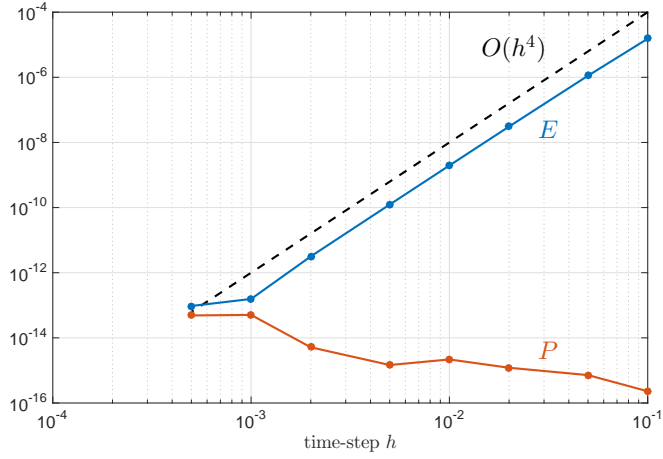}
\caption{\textit{Variation of the relative error $E$ at $t=1$ and of the error in the pole condition $P$ with time-step $h$ for $m=n=128$.
The error $E$ scales as $\mathcal{O}(h^4)$.
Note that the initial condition $Y_{64}^{64}(\lambda,\theta)$ corresponds to the highest wavenumbers that can be resolved on a $128\times 128$ grid.}}
\label{fig:heat}
\end{figure}

As we mentioned in the introduction, Orszag~\cite{orszag1974} and Boyd~\cite{boyd1978} showed that when solving a \textit{time-dependent} PDE 
involving the Laplacian with half-range cosine/sine series \`a la \reff{Orszag}, it is crucial to impose the pole conditions \reff{polecondition2}; 
otherwise, they observed that ``the numerical solution still converged as the time step was shortened---to a wildly wrong answer.''
Therefore, let us now test the accuracy of~\reff{Laplacian} with a time-dependent PDE and illustrate that the pole conditions \reff{polecondition1} are
also satisfied in this case.

We consider the \textit{heat equation} with a thermal diffusivity and an initial condition that lead to a particularly simple exact solution,
\begin{equation}
u_t = \frac{1}{l(l+1)}\Delta u, \quad u(t=0,\lambda,\theta) = Y_{l}^{m}(\lambda,\theta), \quad (\lambda,\theta)\in[-\pi,\pi]\times[0,\pi].
\label{Heat}
\end{equation}

\noindent The exact solution is $u^{ex}(t,\lambda,\theta) = e^{-t}Y_{l}^{m}(\lambda,\theta)$.
Using the DFS method, we seek a solution $\tilde{u}$ of the ``doubled-up'' version of \reff{Heat},
\begin{equation}
\tilde{u}_t = \frac{1}{l(l+1)}\Delta \tilde{u}, \quad \tilde{u}(t=0,\lambda,\theta) = Y_{l}^{m}(\lambda,\theta), \quad (\lambda,\theta)\in[-\pi,\pi]^2,
\label{Heat2}
\end{equation}

\noindent of the form
\begin{equation}
\tilde{u}(t, \lambda, \theta) \approx \sum_{j=-m/2}^{m/2}{\hspace{-0.3cm}}'{\;\,}\sum_{k=-n/2}^{n/2}{\hspace{-0.3cm}}'{\;\,}\hat{u}_{jk}(t)e^{ij\theta}e^{ik\lambda}.
\label{trigsol}
\end{equation}

\noindent We discretize the Laplacian operator with~\reff{Laplacian} and use the fourth-order backward differentiation formula to march in time.
We take $l=64$, $\tilde{u}(t=0,\lambda,\theta)=Y_{64}^{64}(\lambda,\theta)$, $m=n=128$ grid points and solve~\reff{Heat2} up to $t=1$ for various time-steps $h$.
We plot the relative $L^2$-error $E$ at $t=1$,
\begin{equation}
E = \frac{||u(t=1,\lambda,\theta) - u^{ex}(t=1,\lambda,\theta)||_2}{||u^{ex}(t=1,\lambda,\theta)||_2},
\label{error_Heat}
\end{equation}

\noindent and the error in the pole conditions (as defined in~\reff{error_polecondition}) against $h$; the error $E$ scales as $\mathcal{O}(h^4)$, see Figure~\ref{fig:heat}.
(With $m=n=128$ grid points, the error due to the spatial discretization is small compared to the error due to the time discretization, so we are really measuring the latter.)

%%%%%%%%%%%%%%%%%%%%%%%%%%%%%%%%%%%%%%%%%%%%%%%%%%%%%%%%%%%%%%%%%%%%%%%%%%%
\subsection{Bound for the eigenvalues of the Laplacian matrix}

As mentioned in Section~2.4, the eigenvalues of the Laplacian operator are $-l(l+1)$, for integers $l\geq 0$. What about the eigenvalues of the Laplacian matrix~\reff{Laplacian}?
We show in Appendix~A that they are all real and nonpositive, and have observed numerically that some of them are spectrally accurate approximations to the eigenvalues $-l(l+1)$, but some others, the so-called \textit{outliers}~\cite[Ch.~10]{trefethen2000}, are of order $\mathcal{O}(n^2m^2)$ as $n=m\rightarrow\infty$. We shall prove the latter fact below. These large eigenvalues are meaningless physically but of crucial importance in practice since for time-stepping algorithms applied to \reff{ODE} to be stable, we need the eigenvalues of \reff{Laplacian}, scaled by the time-step, to lie in their stability region.

We examine now the largest (in magnitude) eigenvalues of~\reff{Laplacian}. It suffices to examine each block 
\begin{equation}
\Lbf_i = \big(\Dbf_m^{(2)} + \Tbf_{\sin^2}^{-1}\Tbf_{\cos\sin}\Dbf_m + \Dbf^{(2)}_n(i,i)\Tbf_{\sin^2}^{-1}\big), 
\end{equation}

\noindent whose largest eigenvalue can be bounded as 
\begin{equation}
|\lambda_{\max}(\Lbf_i)|\leq \|\Lbf_i\|\leq \|\Dbf_m^{(2)}\| + \|\Tbf_{\sin^2}^{-1}\|\|\Tbf_{\cos\sin}\Dbf_m + \Dbf^{(2)}_n(i,i)\Ibf_m\|.
\end{equation}

\noindent We trivially have $\|\Dbf_m^{(2)}\|=\mathcal{O}(m^2)$ and 
\begin{equation}
\|\Tbf_{\cos\sin}\Dbf_m + \Dbf^{(2)}_n(i,i)\Ibf_m\|
\leq \|\Tbf_{\cos\sin}\Dbf_m\| + |\Dbf^{(2)}_n(i,i)|
=\mathcal{O}(m+i^2).
\end{equation}

\noindent It remains to bound $\|\Tbf_{\sin^2}^{-1}\|$; we claim that this is  $\mathcal{O}(m^2)$. 
To verify this, we come back to the definition of $\Tbf_{\sin^2} = \Qbf\Mbf_{\sin^2}(:,3:m+3)\Pbf$ from \reff{Tsin2_a},
and note that
\begin{equation}
\Qbf^T= \begin{pmatrix}
\mathbf{0}_{2\times m}\\ \Pbf \\ \mathbf{0}_{2\times m}    
\end{pmatrix}
\mrm{diag}(2,1,\ldots,1).
\end{equation}

\noindent We can then write
\begin{equation}
\Tbf_{\sin^2}=\Qbf\Mbf_{\sin^2}\Qbf^T\mbox{diag}(\frac{1}{2},1,\ldots,1)=\mbox{diag}(\sqrt{2},1,\ldots,1)\tilde\Qbf\Mbf_{\sin^2}\tilde\Qbf^T\mbox{diag}(\frac{1}{\sqrt{2}},1,\ldots,1), 
\end{equation}

\noindent where $\tilde\Qbf := \mrm{diag}(\frac{1}{\sqrt{2}},1,\ldots,1)\Qbf$ has orthonormal rows.
Hence, using the fact that $\sigma_{\min}(\mathbf{ABC})\geq \sigma_{\min}(\mathbf{A})\sigma_{\min}(\mathbf{B})\sigma_{\min}(\mathbf{C})$ (which holds when $\mathbf{A}$ and $\mathbf{C}$ are square), we obtain
\begin{equation}
\sigma_{\min}(\Tbf_{\sin^2}) \geq \frac{1}{\sqrt{2}}\sigma_{\min}(\tilde\Qbf\Mbf_{\sin^2}\tilde\Qbf^T). 
\end{equation}

\noindent Now, since $\Mbf_{\sin^2}$ is symmetric positive definite, we have 
$\sigma_{\min}(\tilde\Qbf\Mbf_{\sin^2}\tilde\Qbf^T) =\lambda_{\min}(\tilde\Qbf\Mbf_{\sin^2}\tilde\Qbf^T) 
\geq \lambda_{\min}(\Mbf_{\sin^2}(3:m+3,3:m+3))$, where we used the nonzero structure of $\tilde\Qbf$ for the last inequality. 
Moreover, the eigenvalues of $\Mbf_{\sin^2}(3:m+3,3:m+3)$ are explicitly known to be 
\begin{equation}
\lambda_j= \Big\{\frac{1}{2}\big(\cos(\pi j/(m/2+1))+1\big),\,1\leq j\leq \frac{m}{2}\Big\} 
\cup\Big\{\frac{1}{2}\big(\cos(\pi j/(m/2+2))+1\big),\,1\leq j\leq \frac{m}{2}+1\Big\}.
\label{eq:cluster}
\end{equation}

\noindent 
Thus $1/\sigma_{\min}(\tilde\Qbf\Mbf_{\sin^2}\tilde\Qbf^T) \leq 1/\min_j\lambda_j=\mathcal{O}(m^2)$, and hence $\|\Tbf_{\sin^2}^{-1}\|=1/\sigma_{\min}(\Tbf_{\sin^2})=\mathcal{O}(m^2)$, as required.
We conclude that  $\|\Lbf_i\|=\mathcal{O}(m^2(i^2+m))$; since this holds for every $i$, it follows that 
\begin{equation}
\label{eq:Leig}
|\lambda_{\max}(\Lbf)|=\mathcal{O}(n^2m^2+m^3). 
\end{equation}

\noindent This bound scales as $\mathcal{O}(n^2m^2)$ when $m=n$ (our usual choice).
We illustrate \reff{eq:Leig} in Figure~\ref{fig:maxeig}, which suggests that it is sharp.

The fact that the largest eigenvalue has order $n^2m^2$ makes time-dependent PDEs on the sphere particularly stiff.
It is a consequence of both the second order of the Laplacian operator and the clustering of the points near the poles in~\reff{eq:cluster}.
It would imply severe restrictions on the time-steps for generic explicit algorithms---this is why we use exponential integrators and IMEX schemes, which we describe next.
(In the literature, the severe time-stepping restrictions due to uniform longitude-latitude grids is sometimes called the \textit{pole problem}~\cite{boyd1978, orszag1974}).
Note that the time-stepping restrictions resulting from the clustering near the poles can be addressed by truncating high-frequency terms in the space discretization 
(see, e.g.,~\cite[Sec.~2.1.6]{fornberg1998} and \cite{fornberg1997}). 
However, this approach does not overcome stiffness resulting from the second-order operator. 

\begin{figure}
\centering
\includegraphics[scale=.4]{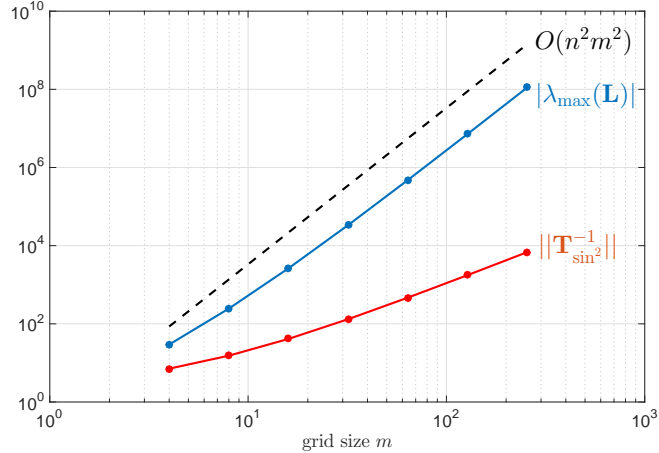}
\caption{\textit{Variation of $|\lambda_{\max}(\Lbf)|$ with $m=n$.
The bound~\eqref{eq:Leig} is accurately reflected.}}
\label{fig:maxeig}
\end{figure}

%%%%%%%%%%%%%%%%%%%%%%%%%%%%%%%%%%%%%%%%%%%%%%%%%%%%%%%%%%%%%%%%%%%%%%%%%%%
\section{Fourth-order time-stepping on the sphere}

Using the DFS method, we seek a solution $\tilde{u}$ of the ``doubled-up'' version of \reff{PDE2},
\begin{equation}
\tilde{u}_t = \alpha\Delta\tilde{u} + \mathcal{N}(\tilde{u}), \quad \tilde{u}(t=0,\lambda,\theta)=\tilde{u}_0(\lambda,\theta), \quad  (\lambda,\theta)\in[-\pi,\pi]^2,
\label{PDE3}
\end{equation}

\noindent of the form~\reff{trigsol}.
Discretizing the Laplacian operator with the Laplacian matrix~\reff{Laplacian}, we obtain the system of $nm$ ODEs \reff{ODE} where $\Lbf$ is \reff{Laplacian} multiplied by $\alpha$.
Time is discretized with time-step $h$ and the problem is to find the Fourier coefficients $\hat{u}^{n+1}$ of $\tilde{u}$ at $t_{n+1}=(n+1)h$ from the coefficients $\hat{u}^{n}$ at $t_{n}=nh$ and also coefficients at previous time-steps (for multistep schemes).
Note that, in practice, nonlinear evaluations $\Nbf(\hat{u}^n)$ are carried out in value space.

We present in this section four time-stepping algorithms for solving \reff{ODE}, and show how it is possible to achieve $\mathcal{O}(nm\log nm)$ complexity per time-step in most cases.
Two of them are exponential integrators based on the ETDRK4 scheme with different strategies for computing the matrix exponential and related functions, while the two others are IMEX schemes.
As before, we have observed numerically that these schemes combined with our spatial discretization preserve both the ``doubled-up'' symmetry \reff{DFS} 
and the pole conditions \reff{polecondition1}, i.e., if one starts with a smooth ``doubled-up'' initial condition, then the solution at time $t$ is also a smooth ``doubled-up'' function.

%%%%%%%%%%%%%%%%%%%%%%%%%%%%%%%%%%%%%%%%%%%%%%%%%%%%%%%%%%%%%%%%%%%%%%%%%%%
\subsection{Exponential integrators}

Dozens of exponential integration formulas of order four and higher have been proposed over the last 15 years~\cite{cox2002, hochbruck2005, hochbruck2010, krogstad2005, luan2014a, minchev2004, ostermann2006}.
The first author recently demonstrated~\cite{montanelli2016c} that it is hard to do much better than the ETDRK4 scheme of Cox and Matthews~\cite{cox2002}. The formula for this scheme is:
\begin{equation}
\begin{array}{l}
\hat{a}^n = \ph_0(h\Lbf/2)\hat{u}^n + (h/2)\ph_1(h\Lbf/2)\Nbf(\hat{u}^n), \\\\
\hat{b}^n = \ph_0(h\Lbf/2)\hat{u}^n + (h/2)\ph_1(h\Lbf/2)\Nbf(\hat{a}^n), \\\\
\hat{c}^n = \ph_0(h\Lbf/2)\hat{a}^n + (h/2)\ph_1(h\Lbf/2)\big[2\Nbf(\hat{b}^n) - \Nbf(\hat{u}^n)\big], \\\\
\hat{u}^{n+1} = \ph_0(h\Lbf)\hat{u}^n + hf_1(h\Lbf)\Nbf(\hat{u}^n) + hf_2(h\Lbf)\big[\Nbf(\hat{a}^n) + \Nbf(\hat{b}^n)\big] + hf_3(h\Lbf)\Nbf(\hat{c}^n),
\end{array}
\label{ETDRK4}
\end{equation}

\noindent where the $\ph$-functions are defined by
\begin{equation}
\begin{array}{l}
\ph_0(h\Lbf) = e^{h\Lbf}, \\\\
\ph_1(h\Lbf) = h^{-1}\Lbf^{-1}(e^{h\Lbf} - \Ibf), \\\\
\ph_2(h\Lbf) = h^{-2}\Lbf^{-2}(e^{h\Lbf} - h\Lbf - \Ibf), \\\\
\ph_3(h\Lbf) = h^{-3}\Lbf^{-3}(e^{h\Lbf} - h^2\Lbf^2/2 - h\Lbf - \Ibf),
\end{array}
\end{equation}

\noindent and the coefficients $f_1$, $f_2$ and $f_3$ are linear combinations of the $\ph$-functions,
\begin{equation}
\begin{array}{l}
f_1(h\Lbf) = \ph_1(h\Lbf) - 3\ph_2(h\Lbf)  + 4\ph_3(h\Lbf) = h^{-3}\Lbf^{-3}[-4\Ibf - h\Lbf + e^{h\Lbf}(4 - 3h\Lbf + (h\Lbf)^2)], \\\\
f_2(h\Lbf) = 2\ph_2(h\Lbf)  - 4\ph_3(h\Lbf) = 2h^{-3}\Lbf^{-3}[2\Ibf + h\Lbf + e^{h\Lbf}(-2\Ibf + h\Lbf)], \\\\
f_3(h\Lbf) = -\ph_2(h\Lbf)  + 4\ph_3(h\Lbf) =  h^{-3}\Lbf^{-3}[-4\Ibf - 3h\Lbf - (h\Lbf)^2 + e^{h\Lbf}(4\Ibf - h\Lbf)].
\end{array}
\end{equation}

When all the eigenvalues of $\Lbf$ are real (i.e., $\alpha\in\R$, $\alpha>0$ corresponding to a diffusive PDE), matrix-vector products $\ph_l(h\Lbf)v$ can be evaluated using rational approximations computed by the Carath\'{e}odory--Fej\'{e}r (CF) method, as described in~\cite{schmelzer2007}.
If $\Lbf$ has some imaginary eigenvalues---the extreme case being when all the eigenvalues are imaginary (i.e., $\alpha\in i\R$, dispersive PDE)---methods based on rational approximations necessarily become expensive, and the $\ph$-functions have to be precomputed before the time-stepping starts, e.g., using the eigenvalue decomposition of $\Lbf$ and contour integrals.\footnote{A comparison of methods for computing the $\ph$-functions can be found in~\cite{ashi2009}.}
We denote by ETDRK4-CF the method with CF approximations and by ETDRK4-EIG the method with eigenvalue decomposition.
We shall give details about the precomputation of the coefficients for both ETDRK4-CF and ETDRK4-EIG below.

Note that Du and Zhu computed in \cite{du2005} the stability region of \reff{ETDRK4} and showed that it includes parts of both the negative real axis and the imaginary axis.

%%%%%%%%%%%%%%%%%%%%%%%%%%%%%%%%%%%%%%%%%%%%%%%%%%%%%%%%%%%%%%%%%%%%%%%%%%%
\subsubsection{ETDRK4-CF}

When all the eigenvalues of $\Lbf$ are real, one can compute matrix-vector products $\ph_l(h\Lbf)v$ in $\mathcal{O}(mn)$ operations
using near-best rational approximations to the $\varphi$-functions on the negative real axis~\cite{schmelzer2007}. 
For an efficient implementation, we use the algorithm that uses common poles for approximating the exponential and other $\varphi$-functions~\cite{schmelzer2007}, as we summarize below. 
Using the CF method for the negative real line \cite{trefethen1983}, we obtain a rational approximant 
\begin{equation}
e^z\approx r_\infty + \sum_{j=1}^p \frac{c_j}{z-z_j}, 
\label{eq:cfapprox}
\end{equation}
  
\noindent which has error decaying like $\approx 9.28903^{-p}$ with a type $(p,p)$ function. 
(We use the MATLAB \texttt{cf} code of Trefethen, Weideman and Schmelzer~\cite{trefethen2006} in our experiments.)
To obtain an approximant to $\varphi_l(z)$ using the approximant~\eqref{eq:cfapprox} to $e^z=\varphi_0(z)$, we use the fact~\cite[Prop.~4.1]{schmelzer2007} that defining 
\begin{equation}
B_z=
\begin{pmatrix}
z&1\\0& 0  
\end{pmatrix}
\end{equation}

\noindent we have 
\setlength{\extrarowheight}{5pt}
\begin{equation}
\varphi_l(B_z)=
\begin{pmatrix}
\varphi_l(z)&\varphi_{l+1}(z)\\0& \varphi_l(0)  
\end{pmatrix}, \quad l\geq 0.
\end{equation}
\setlength{\extrarowheight}{0pt}

\noindent Together with the identity
\setlength{\extrarowheight}{5pt}
\begin{equation}
(B_z-z_jI)^{-1}=
\begin{pmatrix}
(z-z_j)^{-1}&(z-z_j)^{-1}z_j^{-1}
\\0& -z_j^{-1}
\end{pmatrix},
\end{equation}
\setlength{\extrarowheight}{0pt}

\noindent we obtain the approximation
\begin{equation}
\varphi_{l}(z)\approx \sum_{j=1}^p \frac{c_jz_j^{-l}}{z-z_j}, \quad l\geq 0.
\label{eq:phiapp}
\end{equation}

\noindent As suggested in \cite[Prop.~4.1]{schmelzer2007}, we further incorporate a shift $1$ in~\reff{eq:cfapprox}, which with $p=12$ (the default choice; the accuracy is $\approx 10^{-8}$ with $p=10$) gives accuracy $\approx 10^{-10}$ on the negative real axis for all $\ph_l$, $0\leq l\leq 3$. 
Given~\eqref{eq:cfapprox} and~\eqref{eq:phiapp}, evaluating $\ph_l(h\Lbf)b$ for a vector $b$ as in~\reff{ETDRK4} can be approximated as 
\begin{equation}
\ph_l(h\Lbf)b\approx \sum_{j=1}^p c_jz_j^{-l} (h\Lbf-z_j\Ibf)^{-1}b, \quad 0\leq l\leq 3,
\label{action}
\end{equation}

\noindent which reduces to $p=12$ shifted linear systems of the form~\eqref{eq:shiftlinsysL}, which we do with linear cost as described in Section~2.3.
In practice, we compute and store the LU factorizations of the matrices that appear in~\reff{action} before the time-stepping starts.
Note that computing different products $\ph_l(hL)v$ at once with the same $v$ requires no further linear systems. 

Let us emphasize three aspects of this approach.
First, it is not necessary to explicitly compute and store the $\ph$-functions; instead, their action on vectors is directly computed via \reff{action}.
Second, the most expensive operation in~\reff{ETDRK4}--\reff{action} is the 2D FFT, which costs $\mathcal{O}(nm\log nm)$ operations;
see Table~\ref{tab:costs}.
Third, this method is not applicable when $\Lbf$ has some imaginary eigenvalues; 
this is because low-degree rational functions are unable to approximate the exponential on the imaginary axis, which is oscillatory. 
In this case one has to compute and store the $\ph$-functions, and the complexity increases to $\mathcal{O}(nm^2)$ per time-step, as we describe next. 

\begin{figure}
\centering
\includegraphics[scale=.4]{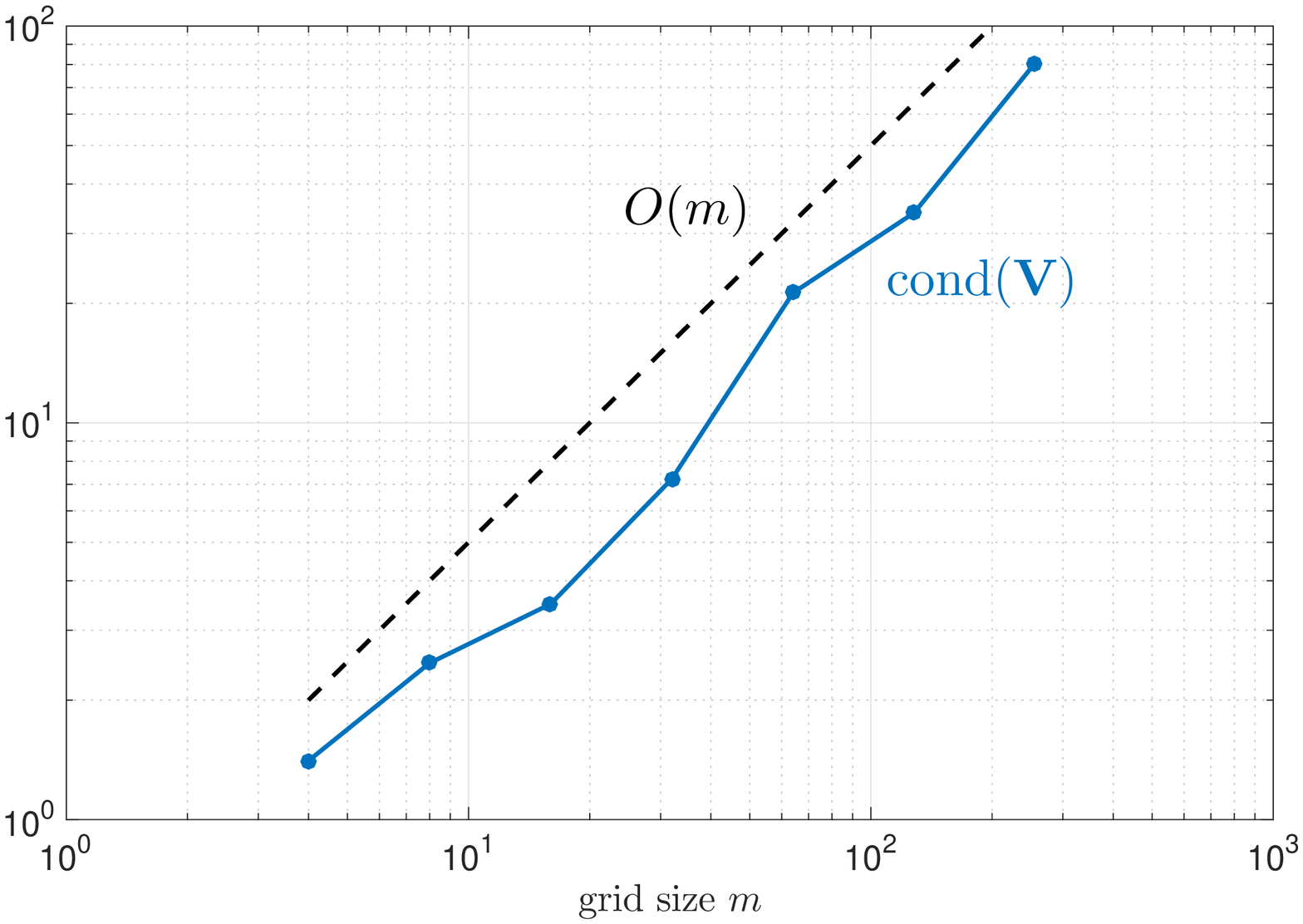}
\caption{\textit{Variation of $\mrm{cond}(\mathbf{V})$ with $m=n$. It is of order $m$ and reasonably small; therefore, an approach 
based on eigenvalues and eigenvectors for the computation of the $\ph$-functions is valid.}}
\label{fig:conds}
\end{figure}

%%%%%%%%%%%%%%%%%%%%%%%%%%%%%%%%%%%%%%%%%%%%%%%%%%%%%%%%%%%%%%%%%%%%%%%%%%%
\subsubsection{ETDRK4-EIG}

To compute the $\ph$-functions, one can use a method based on eigenvalues and eigenvectors. 
The idea is to diagonalize $\Lbf=\mathbf{V\Lambda V^{-1}}$ and then apply the $\varphi$-functions to the eigenvalues, 
\begin{equation}
\ph_l(h\Lbf)=\mathbf{V}\ph_l(h\mathbf{\Lambda})\mathbf{V^{-1}}, \quad 0\leq l \leq 3,
\label{EIG}
\end{equation}

\noindent with 
\begin{equation}
\ph_l(h\mathbf{\Lambda}) = 
\begin{pmatrix}
\ph_l(h\lambda_1) \\
& \ph_l(h\lambda_2) \\
& & \ddots \\ 
& & & \ph_l(h\lambda_{nm})\\
\end{pmatrix}.
\end{equation}

\noindent For a general $nm \times nm$ matrix $\Lbf$, this would require $\mathcal{O}(n^3m^3)$ operations, but for~\reff{Laplacian}, this can be done blockwise in $\mathcal{O}(nm^3)$ operations.
Note that this corresponds to Method 14 of \cite{moler2003}.
Theoretically, this approach only works when $\mathcal{L}$ is nondefective, that is, when it has a complete set of linearly independent
eigenfunctions---this is a well known result for the Laplacian operator on the sphere~\cite{atkinson2012}.
In practice, difficulties occur when the discretizaion $\Lbf$ is ``nearly'' defective, i.e., when $\mrm{cond}(\mathbf{V})=||\mathbf{V}||\,||\mathbf{V}^{-1}||$ is large.
Fortunately, we have observed numerically that the condition number is small and of order $m$ as $m=n$ increases; see Figure \ref{fig:conds}.

Once we have computed the eigenvalue decomposition, we follow the idea of Kassam and Trefethen~\cite{kassam2005} and evaluate the $\varphi$-functions at each scaled eigenvalue $h\lambda$ using Cauchy's integral formula,
\begin{equation}
\ph_l(h\lambda) = \frac{1}{2\pi i}\oint_\Gamma \frac{\ph_l(z)}{z - h\lambda} dz \approx \frac{1}{M} \sum_{k=1}^M \ph_l\big(h\lambda + e^{2\pi i (k-0.5)/M}\big), \quad 0\leq l\leq 3.
\label{contour}
\end{equation}

\noindent The constant $M$ is the number of points in the discretized contour integration; we take $M=32$ in our experiments.

The precomputation step costs $\mathcal{O}(nm^3)$ operations, while the cost per time-step is $\mathcal{O}(nm^2)$ since one has to compute block diagonal matrix-vector products $\ph_l(h\Lbf)v$; see Table~\ref{tab:costs}.

%%%%%%%%%%%%%%%%%%%%%%%%%%%%%%%%%%%%%%%%%%%%%%%%%%%%%%%%%%%%%%%%%%%%%%%%%%%
\subsection{Implicit-explicit schemes}

We present in this section the two IMEX schemes we consider in this paper. 
The first one, IMEX-BDF4~\cite{ascher1995}, is a multistep scheme which is stable only for diffusive PDEs.
The second one, LIRK4~\cite{calvo2001}, is a one-step scheme, stable for both diffusive and dispersive PDEs.

%%%%%%%%%%%%%%%%%%%%%%%%%%%%%%%%%%%%%%%%%%%%%%%%%%%%%%%%%%%%%%%%%%%%%%%%%%%
\subsubsection{IMEX-BDF4}

Following Kassam and Trefethen~\cite{kassam2005}, we consider a scheme known either as SBDF4 (in \cite{ascher1995}), AB4BD4 (in \cite{cox2002}) or IMEX-BDF4 (in \cite{hundsdorfer2007}), which combines a fourth-order Adams--Bashforth formula and a fourth-order backward differentiation scheme in a nontrivial way. The method is given by:
\begin{equation}
\begin{array}{l}
(25\Ibf_{nm} - 12h\Lbf)\hat{u}^{n+1} = 48\hat{u}^{n} - 36\hat{u}^{n-1} + 16\hat{u}^{n-2} - 3\hat{u}^{n-3} + 48h\Nbf(\hat{u}^{n}) - 72h\Nbf(\hat{u}^{n-1})\\\\
\hspace{4cm}  + \; 48h\Nbf(\hat{u}^{n-2}) - 12h\Nbf(\hat{u}^{n-3}).
\end{array}
\label{IMEX-BDF4}
\end{equation}

\noindent At each time-step, one has to solve a linear system to get the Fourier coefficients $\hat{u}^{n+1}$,
which we do with linear cost by multiplying each block of~\reff{IMEX-BDF4} by $\Tbf_{\sin^2}$, as explained in Section 2.3. 
Therefore, the dominant cost in~\reff{IMEX-BDF4} is the $\mathcal{O}(nm\log nm)$  2D FFT for the nonlinear evaluations;
see Table~\ref{tab:costs} at the end of this section.

Let us add three comments about \reff{IMEX-BDF4}.
First, the LU factorization of the left-hand side of \reff{IMEX-BDF4} is computed and stored before the time-stepping starts.
Second, this is a multistep formula so it has to be started with a one-step scheme---in the numerical comparisons of Section~4, we initialize it with three steps of ETDRK4-CF.
Third, it is unstable for dispersive PDEs since the stability region of the fourth-order backward differentiation formula does not contain the portion of the imaginary axis near the origin.
In these cases, one can use IMEX Runge--Kutta schemes~\cite{calvo2001, kennedy2003, pareschi2005} or extrapolation-based
IMEX schemes~\cite{cardone2014, constantinescu2010}. We have decided to focus on the former.

\begin{table}
\caption{\textit{Computational costs (per time-step) of the time-stepping algorithms with $p=12$ in the CF approximation.
The IMEX-BDF$\,4$ scheme is particularly cheap while for ETDRK$\,4$-CF one needs to solve an extremely large number of linear systems.
ETDRK$\,4$-EIG is the only scheme that has a $\mathcal{O}(nm^2)$ cost per time-step and a $\mathcal{O}(nm^3)$ precomputation step.
Precomputations for the other schemes (LU factorizations) cost $\mathcal{O}(nm)$ operations.}}
\vspace{.5cm}
\centering
\ra{1.3}
\begin{tabular}{lccccc}
\hline
& \multicolumn{2}{c}{\textbf{ETDRK4}} & \phantom{a} & \multicolumn{2}{c}{\textbf{IMEX}}\\
\cmidrule{2-3} \cmidrule{5-6} 
& \textbf{CF} & \textbf{EIG} && \textbf{BDF4} & \textbf{LIRK4}\\
\midrule
\textbf{\# $\mathcal{O}(nm\log nm)$ FFTs} & 8 & 8 && 2 & 12\\
\textbf{\# $\mathcal{O}(nm)$ linear solves} & $9p=108$ & 0 && 1 & 5\\
\textbf{\# $\mathcal{O}(nm^2)$ matrix-vector products} & 0 & 9 && 0 & 0\\
\textbf{diffusive PDEs} & \checkmark & \checkmark && \checkmark & \checkmark\\
\textbf{dispersive PDEs} & $\times$ & \checkmark && $\times$ & \checkmark\\
\bottomrule
\end{tabular}
\label{tab:costs}
\end{table}

%%%%%%%%%%%%%%%%%%%%%%%%%%%%%%%%%%%%%%%%%%%%%%%%%%%%%%%%%%%%%%%%%%%%%%%%%%%
\subsubsection{LIRK4}

IMEX Runge--Kutta schemes combine explicit Runge--Kutta formulas to adavance the nonlinear part and implicit Runge--Kutta to advance the linear part~\cite{calvo2001, kennedy2003, pareschi2005}. 
In this paper, we use the fourth-order LIRK4 scheme of Calvo, de Frutos and Novo~\cite{calvo2001}.
It combines an implicit $L$-stable five-stage fourth-order Runge--Kutta method, whose Butcher tableau is given in \cite[Table~6.5]{hairer1991},
with a six-stage fourth-order explicit Runge--Kutta method, whose coefficients were derived in~\cite{calvo2001}. The formula for this scheme is:
\setlength{\extrarowheight}{5pt}
\begin{equation}
\begin{array}{l}
(\Ibf_{nm} - \frac{1}{4}h \Lbf)\hat{a}^n = \hat{u}^n + \frac{1}{4}h\Nbf(\hat{u}^n), \\\\

(\Ibf_{nm} - \frac{1}{4}h \Lbf)\hat{b}^n = \hat{u}^n + \frac{1}{2}h\Lbf\hat{a}^n - \frac{1}{4}h\Nbf(\hat{u}^n) + h\Nbf(\hat{a}^n), \\\\

(\Ibf_{nm} - \frac{1}{4}h \Lbf)\hat{c}^n = \hat{u}^n + \frac{17}{50}h\Lbf\hat{a}^n - \frac{1}{25}h\Lbf\hat{b}^n - \frac{13}{100}h\Nbf(\hat{u}^n) + \frac{43}{75}h\Nbf(\hat{a}^n) + \frac{8}{75}h\Nbf(\hat{b}^n), \\\\

(\Ibf_{nm} - \frac{1}{4}h \Lbf)\hat{d}^n = \hat{u}^n + \frac{371}{1360}h\Lbf\hat{a}^n - \frac{137}{2720}h\Lbf\hat{b}^n + \frac{15}{544}h\Lbf\hat{c}^n - \frac{6}{85}h\Nbf(\hat{u}^n) + \frac{42}{85}h\Nbf(\hat{a}^n), \\\\
\hspace{3.05cm} + \, \frac{179}{1360}h\Nbf(\hat{b}^n) - \frac{15}{272}h\Nbf(\hat{c}^n), \\\\

(\Ibf_{nm} - \frac{1}{4}h \Lbf)\hat{e}^n = \hat{u}^n + \frac{25}{24}h\Lbf\hat{a}^n - \frac{49}{48}h\Lbf\hat{b}^n + \frac{125}{16}h\Lbf\hat{c}^n - \frac{85}{12}h\Lbf\hat{d}^n + \frac{79}{24}h\Nbf(\hat{a}^n) - \frac{5}{8}h\Nbf(\hat{b}^n) \\
\hspace{3.05cm} + \, \frac{25}{2}h\Nbf(\hat{c}^n) - \frac{85}{6}h\Nbf(\hat{d}^n), \\\\

\hat{u}^{n+1} = \hat{u}^n + \frac{25}{24}h\Lbf\hat{a}^n - \frac{49}{48}h\Lbf\hat{b}^n + \frac{125}{16}h\Lbf\hat{c}^n - \frac{85}{12}h\Lbf\hat{d}^n + \frac{1}{4}h\Lbf\hat{e}^n + \frac{25}{24}h\Nbf(\hat{a}^n) - \frac{49}{48}h\Nbf(\hat{b}^n) \\\\
\hspace{1.59cm} + \, \frac{125}{16}h\Nbf(\hat{c}^n) - \frac{85}{12}h\Nbf(\hat{d}^n) + \frac{1}{4}h\Nbf(\hat{e}^n).
\end{array}
\label{LIRK4}
\end{equation}
\setlength{\extrarowheight}{0pt}

The LU factorization of the left-hand sides of~\reff{LIRK4} is computed and stored before the time-stepping starts.
Note that the most expensive operations in the computation of the internal stages $\hat{a}^n$, $\hat{b}^n$, $\hat{c}^n$, $\hat{d}^n$ and $\hat{e}^n$ are the nonlinear evaluations ($\mathcal{O}(nm \log nm)$ work).
The other operations, i.e., linear solves and matrix-vector products in the right-hand side (each block being multiplied by $\Tbf_{\sin^2}$), can be carried out in linear time.
The computation of $\hat{u}^{n+1}$ from $\hat{u}^n$ requires matrix-vector products of the form $\Lbf v$, which reduce to $n$ dense matrix-vector products $\Lbf_iv$; each can be done in $\mathcal{O}(m)$ operations since $\Lbf_i v=\Tbf_{\sin^2}^{-1}(\Tbf_{\sin^2}\Lbf_i)v$. (The LU factorization of $\Tbf_{\sin^2}$ is also computed and stored before the time-stepping starts.)

%%%%%%%%%%%%%%%%%%%%%%%%%%%%%%%%%%%%%%%%%%%%%%%%%%%%%%%%%%%%%%%%%%%%%%%%%%%
\section{Numerical comparisons}

%%%%%%%%%%%%%%%%%%%%%%%%%%%%%%%%%%%%%%%%%%%%%%%%%%%%%%%%%%%%%%%%%%%%%%%%%%%
\subsection{Methodology}

To compare time-stepping schemes, we follow the methodology of~\cite{kassam2005}.
We solve a given PDE up to $t=T$ for various time-steps $h$ and a fixed number of grid points.
We estimate the ``exact'' solution $u^{ex}(t=T,\lambda,\theta)$ by using a very small time-step (half the smallest time-step $h$) and ETDRK4-EIG (the most accurate time-stepping scheme).
We then measure the relative $L^2$-error $E$ at $t=T$ between the computed solution $u(t=T,\lambda,\theta)$ and $u^{ex}(t=T,\lambda,\theta)$,
\begin{equation}
E = \frac{\Vert u(t=T,\lambda,\theta) - u^{ex}(t=T,\lambda,\theta)\Vert_2}{\Vert u^{ex}(t=T,\lambda,\theta)\Vert_2}.
\label{error_PDE}
\end{equation}

\noindent For both $u$ and $u^{ex}$ we use $m=n=256$ grid points.
(With these grid sizes, the error due to the spatial discretization is small compared to the error due to the time discretization.)
We use $p=12$ in the CF approximation for ETDRK4-CF and $M=32$ points to compute the contour integrals~\reff{contour} for ETDRK4-EIG.
We plot \reff{error_PDE} against relative time-steps $h/T$ and computer times on a pair of graphs.\footnote{The precomputation of the coefficients of the exponential integrators, the LU factorizations for the
IMEX schemes and the starting phase of IMEX-BDF4 are not included in the computing time. 
Timings were done on a 2.8\,GHz Intel i7 machine with 16\,GB of RAM using MATLAB R2015b and Chebfun v5.6.0.}
The former gives a measure of the accuracy of the time-stepping scheme for various time-steps or, equivalently, for various number of integration steps.
(If the relative time-step is $10^{-3}$, it means that the scheme performed $10^3$ steps to reach $t=T$.) 
However, it is possible that each step is more costly, so it is the latter that ultimately matters. 

%%%%%%%%%%%%%%%%%%%%%%%%%%%%%%%%%%%%%%%%%%%%%%%%%%%%%%%%%%%%%%%%%%%%%%%%%%%
\subsection{Results for the diffusive case}

The \textit{Allen--Cahn equation}, derived by Allen and Cahn in the 1970s, is a reaction-diffusion equation which
describes the process of phase separation in iron alloys~\cite{allen1979}, studied in the ball and on the sphere in, e.g.,~\cite{du2008}.
It is given by
\begin{equation}
u_t = \epsilon\Delta u + u - u^3, \quad \epsilon\ll1,
\label{eq:AC}
\end{equation}

\noindent with linear diffusion $\epsilon\Delta u$ and a cubic reaction term $u-u^3$.
The function $u$ is the order parameter, a correlation function related to the positions of the different components of the alloy.
The Allen--Cahn equation exhibits stable equilibria at $u=\pm 1$, while $u=0$ is an unstable equilibrium. 
Solutions often display metastability where wells $u\approx -1$ compete with peaks $u\approx 1$,
and structures remain almost unchanged for long periods of time before changing suddenly.
\noindent We take $\epsilon=10^{-2}$ and
\begin{equation}
u(t=0,x,y,z) =  \cos(\cosh(5xz)-10y),
\label{eq:ACIC}
\end{equation}

\noindent and solve up to $t=10$. The initial condition and the solution at times $t=1,2,10$ are shown in Figure~\ref{fig:ACsol}.
The initial condition quickly converges to a metastable $u\approx\pm1$ solution (at around $t=10$) and eventually to the stable constant solution $u=1$
(at around $t=60$).

\begin{figure}
\hspace{-1.1cm}
\includegraphics[scale=.59]{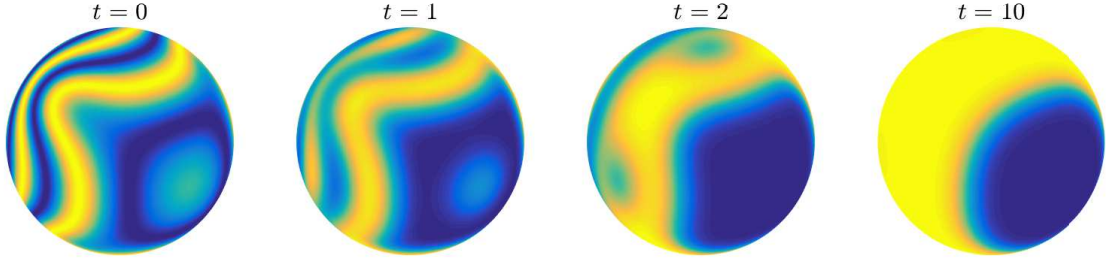}
\vspace{-9.8cm}
\caption{\textit{Initial condition~$\reff{eq:ACIC}$ and solution at times $t=1,2,10$ of the Allen--Cahn equation~$\reff{eq:AC}$.}}
\label{fig:ACsol}
\end{figure}

\begin{figure}
\hspace{-.9cm}
\includegraphics[scale=.45]{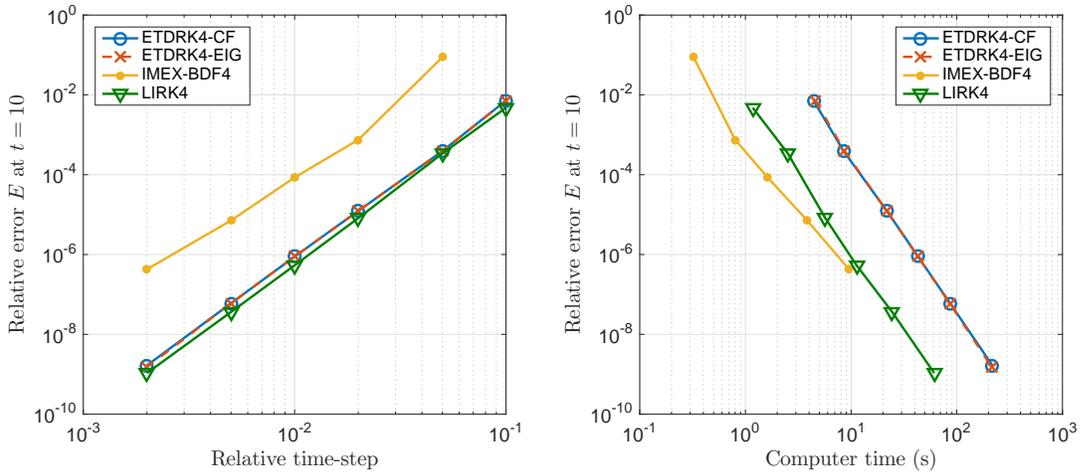}
\caption{\textit{Relative error $E$ at $t=10$ versus relative time-step and computer time for the Allen--Cahn equation~$\reff{eq:AC}$.}}
\label{fig:ACcomp}
\end{figure}

The results are shown in Figure~\ref{fig:ACcomp}. 
All the schemes are stable for the time-steps we have considered, except IMEX-BDF4 for the largest time-step.
The ETDRK4 schemes and LIRK4 have similar accuracy, while IMEX-BDF4 is significantly less accurate.
However, IMEX-BDF4 is the most efficient scheme.
This can be explained by looking at Table~\ref{tab:costs}: IMEX-BDF4 requires very few operations per time-step.
Note that in this experiment, ETDRK4-CF ($\mathcal{O}(nm\log nm)$ work per time-step) is not more efficient than ETDRK4-EIG ($\mathcal{O}(nm^2)$).
Again, the reason can be found in Table~\ref{tab:costs}: ETDRK4-CF requires the solution of 108 linear systems per time-step.
(For $m=n$ sufficiently large, ETDRK4-CF will be indeed more efficient than ETDRK4-EIG.)

%%%%%%%%%%%%%%%%%%%%%%%%%%%%%%%%%%%%%%%%%%%%%%%%%%%%%%%%%%%%%%%%%%%%%%%%%%%
\subsection{Results for the dispersive case}

The \textit{nonlinear Schr\"odinger (NLS) equation},
\begin{equation}
u_t = i\Delta u + i\vert u\vert^2u,
\label{eq:NLS}
\end{equation}

\noindent models a variety of physical phenomena, including the propagation of light in optical fibres; on the sphere, a recent theoretical study of its solutions can be found in~\cite{takaoka2016}.
A nonlinear variant of the Schr\"odinger equation, it couples dispersion $i\Delta u$ with a nonlinear potential $i\vert u\vert^2u$. 
Note that the wave function $u$ is complex-valued.
\noindent We take 
\begin{equation}
u(t=0,\lambda,\theta) = A\bigg(\frac{2B^2}{2-\sqrt{2}\sqrt{2-B^2}\cos(AB\theta)}-1\bigg) + Y_3^3(\lambda,\theta)
\label{eq:NLSIC}
\end{equation}

\noindent with $A=B=1$ and solve up to $t=1$. 
The initial condition and the real part of the solution at times $t=0.3,0.6,1$ are shown in Figure~\reff{fig:NLSsol}.
The initial condition is the superposition of two nonlinear waves in which energy concentrates in a localized and oscillatory fashion, 
a \textit{breather} and a spherical harmonic.

\begin{figure}
\hspace{-1.1cm}
\includegraphics[scale=.57]{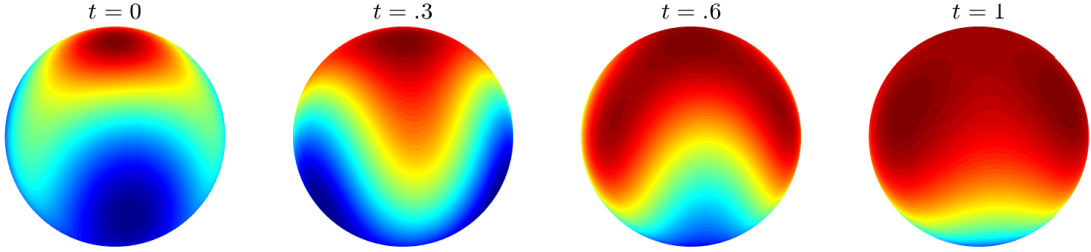}
\vspace{-9.1cm}
\caption{\textit{Initial condition~$\reff{eq:NLSIC}$ and real part of the solution at times $t=0.3,0.6,1$ of the NLS equation~$\reff{eq:NLS}$.}}
\label{fig:NLSsol}
\end{figure}

\begin{figure}
\hspace{-.9cm}
\includegraphics[scale=.45]{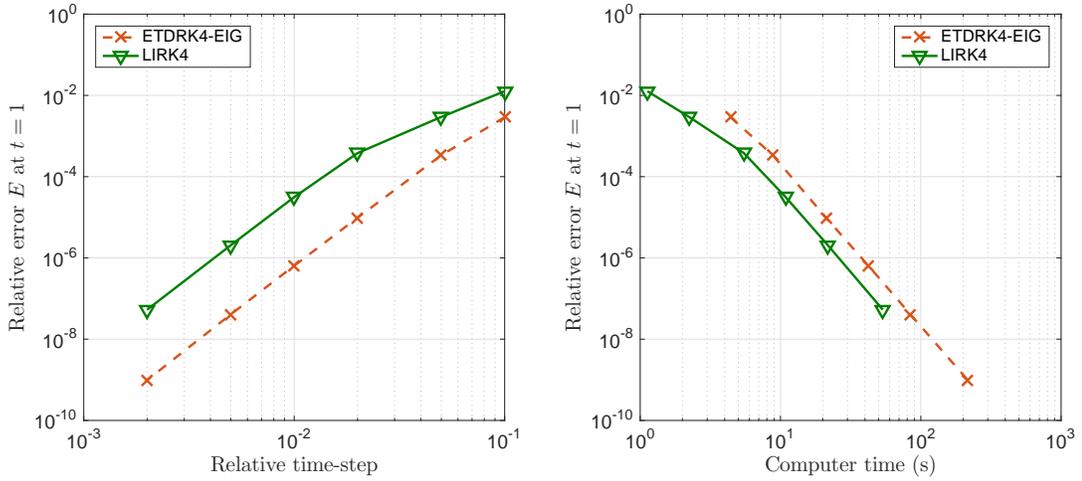}
\caption{\textit{Relative error $E$ at $t=1$ versus relative time-step and computer time for the NLS equation.}}
\label{fig:NLScomp}
\end{figure}

The results are shown in Figure~\ref{fig:NLScomp}. 
Both schemes are stable for the time-steps we have considered.
LIRK4 is less accurate than ETDRK4-EIG in this case, but since it has a much lower cost per time-step, it is more efficient.
Let us emphasize that timings do not include the precomputation step, which takes significantly longer for ETDRK4-EIG ($\mathcal{O}(nm^3)$ versus $\mathcal{O}(nm)$).
Also, since LIRK4 and ETDRK4-EIG have different scaling in $m=n$, the results ultimately depend on the size of the discretization.
For example, for $m=n=128$ instead of $256$, ETDRK4-EIG is slightly more efficient.

%%%%%%%%%%%%%%%%%%%%%%%%%%%%%%%%%%%%%%%%%%%%%%%%%%%%%%%%%%%%%%%%%%%%%%%%%%%
\section{Discussion}

We have presented algorithms for solving stiff PDEs on the sphere with spectral accuracy in space and fourth-order in time. 
For the spatial discretization, we have used a variant of the DFS method in coefficient space.
The main advantages of our method are that it is not necessary to numerically impose the pole conditions~\reff{polecondition1}, 
while operating in coefficient space avoids the coordinate singularity without using shifted grids and leads to matrices $\Lbf$ that can be computed and inverted efficiently.
We have tested our method with the Poisson and heat equations and obtained excellent results.

For solving nonlinear time-dependent PDEs, we have used IMEX schemes and exponential integrators to circumvent the time-stepping restrictions due to the large eigenvalues of the Laplacian matrix.
For diagonal problems, exponential integrators are particularly efficient since the computation and the action of the matrix exponential can trivially be computed in linear time.
For problems that allow for fast numerical linear algebra (fast sparse direct solver), as in this paper, we have given numerical evidence that IMEX schemes outperform exponential integrators. 
The IMEX-BDF4 time-stepping scheme is remarkably efficient for diffusive PDEs but since it is unstable for dispersive PDEs and needs to be started with another algorithm, 
it might be preferable to use LIRK4 for both diffusive and dispersive PDEs. 
For problems that generate dense matrices $\Lbf$, it is not clear which one would perform best. 
On a grid with $N$ points, both exponential integrators (dense matrix-vector products) and IMEX schemes (triangular systems from precomputed LU factorizations) would have a $\mathcal{O}(N^2)$ cost per time-step.
Note that dense matrices correspond to, e.g., the DFS method applied to \reff{PDE} with highly oscillatory variable coefficients or a Chebyshev discretization in value space of PDEs of the form
\reff{PDE} in 1D/2D/3D.\;
However, as recently shown by Aurentz in~\cite{aurentz2015}, it seems that all spectral differentiation matrices (even the dense ones!) might allow for fast numerical linear algebra, which would render IMEX schemes more efficient for value-based Chebyshev discretizations---this is a story to be continued.
We summarize these observations in Table~\ref{tab:comparisons}.

\begin{table}
\caption{\textit{Computation costs per time-step for a grid with $N$ points and most efficient time-stepping scheme in each case.
Diagonal problems were considered in}~\cite{kassam2005}. 
\textit{In this paper, we investigated non-diagonal problems that allow for fast numerical linear algebra (i.e., linear systems can be solved in linear time).
In the latter case, IMEX schemes outperform exponential integrators. 
It is not clear which scheme would perform best in the dense numerical linear algebra case.}}
\vspace{.5cm}
\centering
\ra{1.3}
\begin{tabular}{ccc}
\toprule
& \textbf{diffusive PDEs} & \textbf{dispersive PDEs}\\
\midrule
\textbf{diagonal problems} & $\mathcal{O}(N\log N)$ & $\mathcal{O}(N\log N)$\\
& ETDRK4 & ETDRK4\\\\
\textbf{non-diagonal problems} & $\mathcal{O}(N\log N)$ & $\mathcal{O}(N\log N)$\\
\textbf{fast sparse direct solver} & IMEX-BDF4 & LIRK4\\\\
\textbf{non-diagonal problems} & $\mathcal{O}(N^2)$ & $\mathcal{O}(N^2)$\\
\textbf{dense solver} & TBD & TBD\\
\bottomrule
\end{tabular}
\label{tab:comparisons}
\end{table}

We have not considered the ``sliders'' of Fornberg and Driscoll~\cite{driscoll2002b, fornberg1999}. 
While this can be an efficient alternative to IMEX schemes and exponential integrators for diagonal problems~\cite{kassam2005}, it is not clear how it can be extended to non-diagonal ones.

Our method can in principle deal with nonsmooth initial conditions for diffusive problems, as long as the diffusion is strong enough to smooth out the solution.
Also, as we mentioned in the introduction, our method could be applied to more general PDEs, including linear operators consisting of powers of the Laplacian operator. 
These would have a larger bandwidth, but could still be inverted efficiently. Therefore, we believe that the same conclusions would hold. 
Analogues of the matrices $\Pbf$ and $\Qbf$ would be involved. 

Future directions include the application of our method to hyperbolic problems, e.g., 
the barotropic vorticity equation or the shallow water equations.
Such problems involve nonlinear differential operators with large eigenvalues.
While stiffness in the linear part (as in the Allen--Cahn and NLS equations) can be treated by using IMEX schemes or exponential integrators, it is not obvious how to deal with a stiff nonlinear operator.
For the barotropic vorticity equation,
\begin{equation}
u_t = \mathcal{N}(u) = -\frac{(\Delta^{-1} u)_\theta}{\sin\theta}u_\lambda + \frac{(\Delta^{-1} u)_\lambda}{\sin\theta}(u_\theta - 2\Omega\sin\theta),
\end{equation}

\noindent a possible approach would be to use the EPIRK schemes~\cite{tokman2006}, e.g., the EPIRK2 scheme given by
\begin{equation*}
\hat{u}^{n+1} = \hat{u}^n + \mathbf{J}^{-1}(e^{h\mathbf{J}} - \Ibf)\Nbf(\hat{u}^n), \quad \mathbf{J} = \frac{d\Nbf}{d\hat{u}}(\hat{u}),
\end{equation*}

\noindent with Arnoldi iteration for the matrix-vector products involving the Jacobian matrix $\mathbf{J}$ of the discretization $\Nbf$ of $\mathcal{N}$ in Fourier space---the cost per time-step would also be $\mathcal{O}(N\log N)$ operations.

%%%%%%%%%%%%%%%%%%%%%%%%%%%%%%%%%%%%%%%%%%%%%%%%%%%%%%%%%%%%%%%%%%%%%%%%%%
\begin{appendix}
\section{Eigenvalues of the Laplacian matrix}
The Laplacian matrix~\reff{Laplacian} is a discretized Laplacian, so one might expect that the eigenvalues are all real and nonpositive. 
For example, Gottlieb and Lustman~\cite{gottlieb1983} give a nontrivial proof that the discretization of the second derivative operator in the Chebyshev collocation method on a real interval with separated boundary conditions has real and negative eigenvalues. 
Here we show that essentially the same holds on the sphere for~\reff{Laplacian} (a slight difference is that we have one zero eigenvalue since the Laplacian of a constant is zero). 

\begin{theorem}\label{thm:eigreal}
The eigenvalues of $\Lbf$ in~\eqref{Laplacian} are all real and nonpositive. 
\end{theorem}
\begin{proof}   
Clearly it suffices to examine each $i$th block $\Lbf_i=\Dbf_m^{(2)} + \Tbf_{\sin^2}^{-1}\Tbf_{\cos\sin}\Dbf_m + \Dbf^{(2)}_n(i,i)\Tbf_{\sin^2}^{-1}$. 
We first note that the eigenvalues of $\Lbf_i$ are equal to those of the matrix pencil 
\begin{equation} 
\Abf_i-\lambda \Bbf_i :=\Tbf_{\sin^2}\Dbf_m^{(2)} + \Tbf_{\cos\sin}\Dbf_m + \Dbf^{(2)}_n(i,i)\Ibf_m
-\lambda\Tbf_{\sin^2},
\label{eq:defAB}
\end{equation}
which corresponds to the generalized eigenvalue problem $\Abf_i x=\lambda \Bbf_i x$.
We shall prove that this pencil has negative real eigenvalues, regardless of $i$;  we drop the subscript $i$ for simplicity. 
Our proof proceeds as follows: 
\begin{enumerate}
\item The eigenvalues of $\Abf-\lambda\Bbf$ are the values of $\lambda_0$ for which  the matrix 
$\Abf-\lambda_0\Bbf$ has a zero eigenvalue. 
\item For any fixed $\lambda_0\in (-\infty,0]$, all the eigenvalues of the matrix $\Abf-\lambda_0\Bbf$ are real. Therefore we can define 
$m$ real continuous functions $f_j(\lambda_0):=\lambda_j(\Abf-\lambda_0\Bbf)$ 
for $j=1,\ldots,m$. 
\item For every $j$, we have $f_j(0)\leq 0$, and $f_j(\lambda_0)\geq 0$ for negative $\lambda_0$ with sufficiently large $|\lambda_0|$. 
\item By the intermediate value theorem to each $f_j(\lambda_0)$, there is at least one root $f_j(\lambda_0)=0$ 
in $\lambda_0\in(-\infty,0]$
for each $j$. It follows that $\Abf-\lambda \Bbf$ has $m$ real eigenvalues (counting multiplicities), hence so does $\Lbf_i$. 
\end{enumerate}
The only nontrivial parts are the second step, and the claim $f_j(0)\leq 0$. 

We first prove the second step, that the matrix $\Cbf(\lambda_0):=\Abf-\lambda_0\Bbf$ has only real eigenvalues.
To do this we apply the similarity transformation with the permutation matrix $\Pbf=\Ibf_m(:,[1:2:m, 2:2:m])$, 
which gives $\Pbf^T\Cbf(\lambda_0)\Pbf=\mbox{diag}(\Cbf_1(\lambda_0),\Cbf_2(\lambda_0))$, 
a block-diagonal matrix with two $\frac{m}{2}\times \frac{m}{2}:=\ell\times \ell$ blocks. 
Here $\Cbf_1(\lambda_0)$ is tridiagonal with extra elements in the upper-right and lower-left corners (as in~\eqref{eq:C1} below), 
and $\Cbf_2(\lambda_0)$ is tridiagonal.\footnote{It is possible to use this structure in the linear solvers. 
We did not do this in our experiments, as applying the permutation also requires $\mathcal{O}(nm)$ operations.}

For $\Cbf_2(\lambda_0)$, we can verify that the products of the neighboring off-diagonals are positive,
\begin{equation}
\Cbf_2(\lambda_0)_{j,j+1}\Cbf_2(\lambda_0)_{j+1,j}>0,
\end{equation} 

\noindent for all $j$ (when $\lambda_0=0$ the product can be 0; we exclude this case for the moment and assume $\lambda_0<0$).
Hence, $\Cbf_2(\lambda_0)$ is diagonally similar to a symmetric matrix, thus its eigenvalues are all real.  

It remains to prove that the eigenvalues of $\Cbf_1(\lambda_0)$ are all real.
Note that $\Cbf_1(\lambda_0)$ is of the form
\begin{equation}\label{eq:C1}
\Cbf_1(\lambda_0)
=\begin{pmatrix}
\alpha & \beta' &  & & & & & \beta' \\
\beta & \alpha_1 & \beta_1 &  \\
 & \beta_{\ell-2} & \alpha_2 & \beta_2 &  \\
&  & \beta_{\ell-3} & \ddots & \ddots &  \\
& &  & \ddots & \ddots & \ddots &  \\
& & &  & \ddots & \ddots & \beta_{\ell-1} & \\
 & & & &  & \beta_{2} & \alpha_{2} & \beta_{\ell-2} \\
\beta & & & & &  & \beta_{1} & \alpha_{1}
\end{pmatrix}. 
\end{equation}
Several properties of $\Cbf_1(\lambda_0)$ are worth noting: 
(i) the $(\ell-1)\times (\ell-1)$ submatrix obtained by removing the first row and column 
is symmetric about the antidiagonal, both in the diagonal and off-diagonal elements (note the double appearance of $\beta_i$) and
(ii) the products of the neighboring off-diagonal blocks are all positive $\beta_j\beta_{\ell-j-1}> 0$ for all $j$, and $\beta\beta'>0$. 
In Lemma~\ref{lem:realeig} below we prove that any matrix~\eqref{eq:C1} with such structure has only real eigenvalues, establishing that $f_j(\lambda_0)$ is real for any $\lambda_0<0$. The claim extends to $\lambda_0=0$ by continuity of the eigenvalues, completing the second step in the proof. 

It remains to show $f_j(0)\leq 0$ for every $j$, that is, $\Abf$ has only nonpositive eigenvalues.
Since $\Dbf_n^{(2)}(i,i)\leq 0$ for all $i$, it suffices to treat the case for which $\Dbf_n^{(2)}(i,i)=0$. 
We again examine $\Cbf_1(0)$ and $\Cbf_2(0)$ separately. 
For each of these, after deflating the zero eigenvalue (if present) we can apply a diagonal similarity transformation so that the Gershgorin disks, whose centers lie on the negative half line, do not contain the origin, implying that all the eigenvalues are nonpositive. This completes the proof of Theorem~\ref{thm:eigreal}. 
\end{proof}

It remains to prove that the eigenvalues of matrices of the form in~\eqref{eq:C1} are all real. A key fact is that a real tridiagonal matrix with the neighboring off-diagonals having the same sign is diagonally similar to a symmetric tridiagonal matrix, and an analogous result holds for  arrowhead matrices. 
\begin{lemma}\label{lem:realeig}
For any real matrix of the form~\eqref{eq:C1}, 
with  $\beta_i\beta_{\ell-i-1}> 0$ for all $i$ and 
$\beta\beta'>0$, 
all the eigenvalues are real. 
\end{lemma}
\begin{proof}
Denote the matrix by $\Cbf$. We can apply a diagonal similarity transformation to the bottom-right $(\ell-1)\times(\ell-1)$ part $\Cbf_2$, to obtain a symmetric matrix $\Dbf^{-1}\Cbf_2\Dbf$. 
Since $\Cbf_2$ is symmetric about the antidiagonal, so is the diagonal matrix $\Dbf$; thus 
$\Dbf^{-1}\Cbf_2\Dbf$ is both symmetric and symmetric about the antidiagonal, and 
so the transformation 
$\widehat \Cbf= \big[
\begin{smallmatrix}
1\\ & \Dbf^{-1}  
\end{smallmatrix}
\big]
\Cbf
\big[
\begin{smallmatrix}
1\\& \Dbf  
\end{smallmatrix}
\big]$
preserves the property that the off-diagonal parts of the first row and column are parallel (when one is transposed). 
Now let $\Qbf$ be an orthogonal matrix of eigenvectors of $\Dbf^{-1}\Cbf_2\Dbf$ 
such that $\Qbf^T(\Dbf^{-1}\Cbf_2\Dbf)\Qbf$ is diagonal, and consider the matrix 
$\widetilde \Cbf = \big[
\begin{smallmatrix}
1\\ & \Qbf^T
\end{smallmatrix}
\big]
\widehat \Cbf
\big[
\begin{smallmatrix}
1\\& \Qbf  
\end{smallmatrix}
\big]$. 
By the antidiagonal symmetry of $\Dbf^{-1}\Cbf_2\Dbf$, each eigenvector (column of $\Qbf$) has the property that it is either in the form 
$[v_1,v_2,\ldots,-v_2,-v_1]^T$ or $[v_1,v_2,\ldots,v_2,v_1]^T$.
Therefore, $\widetilde \Cbf$ is an arrowhead matrix with the property that 
for every $j$, we have either $\widetilde \Cbf_{j,1}=\widetilde \Cbf_{1,j}=0$, or $\widetilde \Cbf_{j,1}\widetilde \Cbf_{1,j}>0$.
It follows that there exists a diagonal similarity transformation that brings $\widetilde \Cbf$ to symmetric form, hence has real eigenvalues. 
\end{proof}

\begin{figure}
\begin{footnotesize}
\begin{verbatim}
% Parameters:
  m = 1024; n = m;                                    % number of grid points
  h = 1e-1; T = 100;                                  % time-step and final time
  u0 = @(x,y,z) cos(40*x)+cos(40*y)+cos(40*z);
  th = pi/8; c = cos(th); s = sin(th);
  u0 = 1/3*spherefun(@(x,y,z) u0(c*x-s*z,y,s*x+c*z)); % initial condition
  v0 = reshape(coeffs2(u0, m, n), m*n, 1);            % Fourier coefficients 
  
% Nonlinear operator (evaluated in value space):
  g = @(u) u - (1+1.5i)*u.*(abs(u).^2);               % N(u) = u-(1+1.5)i*u*|u|^2
  c2v = @(u) trigtech.coeffs2vals(u);                 % coeffs to values in 1D
  c2v = @(u) c2v(c2v(reshape(u,m,n)).').';            % coeffs to values in 2D
  v2c = @(u) trigtech.vals2coeffs(u);                 % values to coeffs in 1D
  v2c = @(u) reshape(v2c(v2c(u).').',m*n,1);          % values to coeffs in 2D
  N = @(u) v2c(g(c2v(u)));                            % nonlinear operator       
    
% Construct the Laplacian matrix (multiplied by Tsin2 and 1e-4): 
  Dm = spdiags(1i*[0,-m/2+1:m/2-1]', 0, m, m);                       
  D2m = spdiags(-(-m/2:m/2-1).^2', 0, m, m);
  D2n = spdiags(-(-n/2:n/2-1).^2', 0, n, n);
  Im = speye(m); In = speye(n);
  P = speye(m+1); P = P(:, 1:m); P(1,1) = .5; P(m+1,1) = .5;
  Q = speye(m+1+4); Q = Q(3:m+2,:); Q(1,3) = 1; Q(1,m+3) = 1;
  Msin2 = toeplitz([1/2, 0, -1/4, zeros(1, m+2)]);
  Msin2 = sparse(Msin2(:, 3:m+3));
  Tsin2 = round(Q*Msin2*P, 15);
  Mcossin = toeplitz([0, 0, 1i/4, zeros(1, m+2)]);
  Mcossin = sparse(Mcossin(:, 3:m+3));                     
  Tcossin = round(Q*Mcossin*P, 15);
  Lap = 1e-4*(kron(In, Tsin2*D2m + Tcossin*Dm) + kron(D2n, Im));

% Compute LU factorizations of LIRK4 matrices:
  Tsin2 = kron(In, Tsin2);
  [L, U] = lu(Tsin2); [La, Ua] = lu(Tsin2 - 1/4*h*Lap); 
  
% Time-stepping loop:
  itermax = round(T/h); v = v0;
  for iter = 1:itermax
      Nv = N(v); w = Tsin2*v;
      wa = w + h*Tsin2*1/4*Nv;
      a = Ua\(La\wa); Na = N(a);
      wb = w + h*Lap*1/2*a + h*Tsin2*(-1/4*Nv + Na);
      b = Ua\(La\wb); Nb = N(b);
      wc = w + h*Lap*(17/50*a - 1/25*b) + h*Tsin2*(-13/100*Nv + 43/75*Na + 8/75*Nb);
      c = Ua\(La\wc); Nc = N(c);
      wd = w + h*Lap*(371/1360*a - 137/2720*b + 15/544*c) ...
          + h*Tsin2*(-6/85*Nv + 42/85*Na + 179/1360*Nb - 15/272*Nc);
      d = Ua\(La\wd); Nd = N(d);
      we = w + h*Lap*(25/24*a - 49/48*b + 125/16*c - 85/12*d) ...
          + h*Tsin2*(79/24*Na - 5/8*Nb + 25/2*Nc - 85/6*Nd);
      e = Ua\(La\we); Ne = N(e);
      v = v + h*(U\(L\(Lap*(25/24*a - 49/48*b + 125/16*c - 85/12*d + 1/4*e)))) ...
          + h*(25/24*Na - 49/48*Nb + 125/16*Nc - 85/12*Nd + 1/4*Ne);
  end
  vals = c2v(v);                                      % tramsform to value space
  vals = vals([m/2+1:m 1], :);                        % restrict to [-pi,pi]x[0,pi]
  u = spherefun(real(vals)); plot(u)                  % output real(u) and plot
\end{verbatim}
\end{footnotesize}
\caption{MATLAB \textit{code to solve the Ginzburg--Landau equation on the sphere with the DFS method and the LIRK\,{\nf 4} time-stepping scheme; this code can be used for both diffusive and dispersive PDEs.}}
\label{code:LIRK4}
\end{figure}

\end{appendix}

%%%%%%%%%%%%%%%%%%%%%%%%%%%%%%%%%%%%%%%%%%%%%%%%%%%%%%%%%%%%%%%%%%%%%%%%%%%
\section*{Acknowledgements}

We thank Grady Wright for a fruitful exchange of emails about multiplication matrices and for reading an early draft of this manuscript. 
We also thank Alex Townsend and Heather Wilber for discussions about Fourier series on spheres, Jared Aurentz for various suggestions related to numerical linear algebra, and the referees for their helpful comments.
The authors are much indebted to Nick Trefethen for his continual support and encouragement.

%%%%%%%%%%%%%%%%%%%%%%%%%%%%%%%%%%%%%%%%%%%%%%%%%%%%%%%%%%%%%%%%%%%%%%%%%%%
\bibliographystyle{siam}
\bibliography{spinsphere}

%%%%%%%%%%%%%%%%%%%%%%%%%%%%%%%%%%%%%%%%%%%%%%%%%%%%%%%%%%%%%%%%%%%%%%%%%%%
\end{document}